\documentclass[12 pt]{amsart}
\usepackage{amsmath,amssymb,amsthm}
\usepackage{blkarray}
\usepackage[normalem]{ulem}
\usepackage{rotating}
\usepackage{fullpage}
\usepackage{tikz}
\usepackage{xcolor}
\usepackage{hyperref} 
\usepackage{pbox}
\usepackage{verbatim} 
\usepackage{multicol}
\usepackage[capitalize,nameinlink,noabbrev,nosort]{cleveref}
\usepackage{caption}

\newtheorem{theorem}{Theorem}[section]
\newtheorem{conjecture}[theorem]{Conjecture}
\newtheorem{corollary}[theorem]{Corollary}
\newtheorem{definition}[theorem]{Definition}

\newtheorem{lemma}[theorem]{Lemma}
\newtheorem{problem}[theorem]{Problem}
\newtheorem{proposition}[theorem]{Proposition}
\newtheorem{remark}[theorem]{Remark}

\newcommand{\tG}{\widetilde{G}}
\newcommand{\perc}[1]{\mu^{\text{P}}_{#1, p}}
\newcommand{\rc}[1]{\mu^{\text{RC}}_{#1, p, q}}
\newcommand{\forest}[1]{\mu^{\text{F}}_{#1, \lambda} }
\newcommand{\ust}[1]{\mu^{\text{UST}}_{#1}}
\newcommand{\altperc}[2]{\nu^{\text{P}}_{#1, #2}}

\newcommand{\altforest}[2]{\nu^{\text{F}}_{#1, #2} }

\begin{document}

\title{The bunkbed problem and the random cluster model}

\author{Arvind Ayyer}
\address{Arvind Ayyer, Department of Mathematics, Indian Institute of Science, Bangalore - 560012, India}
\email{arvind@iisc.ac.in}

\author{Svante Linusson}
\address{Svante Linusson, Department of Mathematics, KTH-Royal Institute of Technology, SE-100 44, Stockholm, Sweden}
\email{linusson@math.kth.se}

\author{Mohan Ravichandran}
\address{Mohan Ravichandran, Zessta Software Services, A block, Cyber Gateways, Madhapur, Hyderabad 500081}
\email{mohan.ravichandran@gmail.com}

\date{\today}


\begin{abstract}
The well known bunkbed conjecture about percolation on finite graphs is now resolved; Gladkov, Pak and Zimin, building upon work of Hollom, have constructed a counterexample. We revisit this conjecture and study it in the broader context of the class of random cluster measures. We show that the major partial (positive) results on the bunkbed conjecture can also be proved for all random cluster measures, including the results for complete graphs, complete bipartite graphs, and the case when $p \uparrow 1$. 

The arboreal gas measure for forests is another limit of the random cluster measure for which we conjecture the inequality to be true and provide proofs in special cases. We identify a setting where the conjecture does hold, that of ``almost spanning tree measures''. A further analysis leads to intriguing correlation inequalities that complement Rayleigh's inequalities for spanning tree measures.
\end{abstract}

\maketitle

\section{Introduction}
\label{sec:intro}

Combinatorial discrete probability is notorious for easy to state and plausible conjectures. Some of these are utterly frustrating, with no sign of resolution for decades. Sometimes though, these do get resolved by dint of human ingenuity after a long barren spell. And sometimes, these resolutions are so counterintuitive that the reader swears never to trust their intuition again. An exemplar of this is the bunkbed conjecture due to Kasteleyn~\cite[Remark~5]{vandenberg-kahn-2001}, which predicts that two natural distance metrics on finite graphs have the same qualitative behavior. To describe this conjecture, we first need the concept of bunkbed graphs. 

We will denote graphs by $G = (V, E)$ where $V$ is the vertex set and $E$ is the edge set.
Unless noted otherwise, we will denote $n$ for the number of vertices.

\begin{definition}
The \textit{bunkbed graph} of the graph $G = (V, E)$ is the graph $\tG = (\widetilde{V}, \widetilde{E})$ where 
\[
 \widetilde{V} = \{v_1 \mid v \in V\} \cup (\{v_2 \mid v \in V\},
\]
and 
\[
\widetilde{E} = \{(v_1,v_2) \mid  v \in V\} \cup 
\{(u_i,v_i) \mid (u, v) \in E, \, i \in \{1, 2\}\}.
\]
\end{definition}

See \cref{fig:bunkbed_example} for an example. 

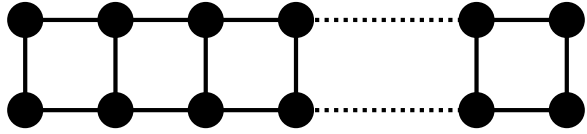
\begin{figure}[h]
    \centering
    \begin{tikzpicture}[scale=1.2]\setlength{\tabcolsep}{0pt}
		\node  (1) at (0, 0) {};
		\node  (2) at (1, 0) {};
		\node  (3) at (2, 0) {};
		\node  (4) at (5, 0) {};
		\node  (5) at (6, 0) {};
  		\node  (1) at (0, 1) {};
		\node  (2) at (1, 1) {};
		\node  (3) at (2, 1) {};
		\node  (4) at (5, 1) {};
		\node  (5) at (6, 1) {};
	
	\draw[-,black,ultra thick] (0, 0) to (1,0) to (2, 0) to (3, 0) to (3, 1) to (2, 1) to (1, 1) to (0, 1) to (0, 0);
    \draw[-,black,ultra thick] (1, 0) to (1, 1);
    \draw[-,black,ultra thick] (2, 0) to (2, 1);
    \draw[black,dotted, ultra thick] (3, 0) to (5,0) ;
    \draw[black,dotted, ultra thick] (3, 1) to (5,1) ;

    \draw[-,black, ultra thick] (5, 0) to (6,0) to (6, 1) to (5, 1) to (5, 0) ;
        \fill(0,0) circle(.2); 
	    \fill(1,0) circle(.2); 
		\fill(2,0) circle(.2); 
		\fill(3,0) circle(.2); 
	    \fill(5,0) circle(.2); 
		\fill(6,0) circle(.2); 
		\fill(0,1) circle(.2); 
	    \fill(1,1) circle(.2); 
		\fill(2,1) circle(.2); 
		\fill(3,1) circle(.2); 
	    \fill(5,1) circle(.2); 
		\fill(6,1) circle(.2);

    \end{tikzpicture}
    \caption{The bunkbed graph of the line graph, the ladder graph}
    \label{fig:bunkbed_example}
\end{figure} 

There are several natural probability measures one can associate to graphs. 
Let $G = (V, E)$ be a graph. 
We will focus on the situation where probability measures will be associated to the edge set $E$.
If $\nu$ is such a probability measure and $A$ is an event, we will denote by {$\nu(A)$} the probability of the event $A$ under $\nu$.
{We will extend this notation for connectivity events as well. Suppose $a,b \in V$. By $\nu(a \leftrightarrow b)$, we will mean the probability that $a$ is connected to $b$ under $\nu$.} 
For us, the most general measure is the random cluster measure, first studied in the doctoral thesis of Fortuin~\cite{fortuin-1971,fortuin-kasteleyn-1972}. 
For any subset $S \subseteq E$, let $\kappa(S)$ denote the number of connected components of $S$.

\begin{definition}
\label{def:rc}
Let $p, q$ be real numbers satisfying $p \in [0, 1]$ and $q > 0$.
Then the \emph{random cluster measure} 
$\rc{G}$
is a probability measure on 
$E$ which is such that for any 
subset 
$S \subset E(G)$, 
\[
\rc{G}(S) \propto p^{|S|} (1-p)^{m - |S|} q^{\kappa(S)}.
\]
\end{definition}

The random cluster measure in \cref{def:rc} simultaneously generalizes several measures on edges of graphs in the literature.
When $q = 1$, this is the \textit{(bond) percolation measure} with parameter $p$ that chooses every edge independently with probability $p$, and which we denote $\perc{G}$. 
When $q = 2$, this is closely related to the very well-studied Ising model on $G$, 
and more generally, when $q$ is a positive integer, it is related to the Potts model on $G$ with $q$ colours~\cite[Section (1.4)]{grimmett-2006}. 

Let $G$ be a graph and $\tG$ the associated bunkbed graph. 
The original bunkbed conjecture was posed by Kasteleyn (see \cite[Remark~5]{vandenberg-kahn-2001}) in the 1980's for the percolation measure and states that for any pair of vertices $u$ and $v$ in $G$, we have that
\begin{align} 
\label{bunkbed}
\perc{\tG} (u_1 \leftrightarrow v_1) \geq \perc{\tG} (u_1 \leftrightarrow v_2). 
\end{align}
This conjecture saw a trickle of positive results over the years \cite{haggstrom-2003, buyer_2019, richthammer_2022, hutchcroft_et_al_2023, hollom_2024}, but was recently sensationally disproved by Gladkov, Pak and Zimin.

\begin{theorem}[{\cite[Theorem 1.2]{gladkov_pak_zimin_2024}}]
\label{thm:bunkbed false}
The bunkbed conjecture is false 
for the percolation measure. Indeed, there is a planar counterexample to this conjecture. 
\end{theorem}

The planar counterexample is carefully chosen to simulate a hypergraph due to Hollom~\cite{hollom_2025}, which was in turn used by him to construct a counterexample to a natural generalization of the bunkbed conjecture to hypergraphs. 

The bunkbed \sout{conjecture} problem is best known in the context of bond percolation, but Kasteleyn himself postulated an analogous statement for the random cluster measure. We formulate it as follows.

\begin{problem}
\label{prob:origRC}
Let $G$ be a graph and $\tG$ the associated bunkbed graph. 
For which values of $p,q$ is it true that for any pair of vertices 
$u$ and $v$ in $G$, we have
\begin{align} 
\label{bunkbed RC}
\rc{\tG}(u_1 \leftrightarrow v_1) \geq \rc{\tG}(u_1 \leftrightarrow v_2)?
\end{align}
\end{problem}

H\"aggstr\"om proved an important case of the bunkbed conjecture for the random cluster measure.

\begin{theorem}[{\cite[Theorem 2.4]{haggstrom-2003}}]
The bunkbed conjecture for the random cluster model is true when $q = 2$, i.e. 
\eqref{bunkbed RC} is true for all graphs and all $p \in [0, 1]$. 
\end{theorem}

As mentioned above, the random cluster measure at $q = 2$ is closely related to the Ising model, and H\"aggstr\"om proved this result using well-understood correlation inequalities for the Ising model, together with a careful conditioning argument. 
We prove the following result in \cref{sec:app} building on the example in~\cite{gladkov_pak_zimin_2024}.

\begin{theorem}
\label{thm:bunkbed rc false}
The bunkbed conjecture for the random cluster measure $\rc{G}$ is not true in general for $0.56 < q < 1.43$.
\end{theorem}

In this article, we take a few steps towards understanding what can hold for other random cluster measures. 
A natural measure on forests can be obtained as a weak limit of the random cluster measure. 
Let $\forest{G}$
be the measure on forests in $G$ such that for a forest $F$, we have that
\[
\forest{G} \propto \lambda^{E(F)}.
\]
This is called the \emph{arboreal gas measure} and was first studied in~\cite{caracciolo-et-al-2004,jacobsen-salas-sokal-2005}.
One can show that~\cite[Equation (1.22)]{grimmett-2006}
\[
\mu^{AG}_{G, \lambda}, 
\rc{G} \Big|_{p = \lambda q} \underset{q \rightarrow 0}{\longrightarrow} \forest{G}, 
\]
in the sense of weak convergence. 
The \textit{uniform spanning tree measure} $\ust{G}$, that assigns equal probability to all spanning trees,
can also be obtained as a weak limit of these measures as follows.
Then we have~\cite[Theorem (1.23)]{grimmett-2006}
\[
\rc{G} \underset{\stackrel{q \rightarrow 0}{q/p \rightarrow 0}}{\longrightarrow} \ust{G}, 
\]
in the sense of weak convergence. 

In contrast to \cref{thm:bunkbed rc false}, we believe that the following conjecture holds.

\begin{conjecture}
\label{conj:forest}
Let $G$ be a graph and $u, v$ be vertices of $G$. 
Then
\begin{align} 
\label{bunkbed forests}
\forest{\tG}(u_1 \leftrightarrow v_1) \geq \forest{\tG}(u_1 \leftrightarrow v_2).
\end{align}
\end{conjecture}

We begin with necessary background on algebraic graph theory, the two-coloured bunkbed model on outerplanar graphs and basics of correlation inequalities in \cref{sec:background}.
We then prove the following set of results complementary to the disproof~\cite{gladkov_pak_zimin_2024}.

\begin{enumerate}
    
\item In recent work, Hutchcroft, Kent and Nizi\'c-Nikolac~\cite{hutchcroft_et_al_2023} showed that the regular bunkbed conjecture holds for any graph in the limit $p \uparrow 1$. A little while later, Hollom~\cite{hollom_2024} gave a different proof of this result. We show  in \cref{sec:pto1} that this is also true for all random cluster measures when we fix $q$ and let $p \uparrow 1$. 
    
\item  The bunkbed conjecture for percolation on the complete graph was proved for $p \geq 1/2$ by Buyer~\cite{buyer_2019} and then for all values of $p$ by van Hintum and Lammers~\cite{van_Hintum_2019}. 
The result for percolation on complete bipartite graphs was proved by Richthammer~\cite{richthammer_2022}. 
We show that the bunkbed conjecture is true for all random cluster measures for the complete and complete bipartite graphs in \cref{sec:complete and bipartite}.
    
\item  In \cref{sec:outerplanar} we show that the bunkbed conjecture for arboreal gas measures  holds for all outerplanar graphs as long as there are at most two posts\footnote{The starting point of this paper was our hope that we could prove \cref{conj:forest}. 
The model is more constrained in this setting and one additionally has tools from statistical physics such as the ``magic formula'' for connection probabilities in the arboreal gas model; see \cite[Equation (1.5)]{bauerschmidt2021random}. }. 

\item In \cref{sec:effec}, we give a proof of this conjecture for the further limiting case of the random cluster measure, where one studies the arboreal gas measure for large $\lambda$. 
    It turns out that this is equivalent to a statement about effective resistances between pairs of vertices in the bunkbed graph: The effective resistance between $u_1$ and $v_1$ is less than that between $u_1$ and $v_2$. This can be proved, as we show, using the matrix tree theorem. Our proof in this regime naturally leads to new correlation inequalities that complement Rayleigh's theorem on spanning tree measures, presented in \cref{sec:corr}.

\item What about the bunkbed conjecture for other values of $q$? The situation here is very murky. In \cref{sec:app}, we revisit the counterexample due to Gladkov--Pak--Zimin~\cite{gladkov_pak_zimin_2024} using Hollom's hypergraph gadget~\cite{hollom_2025} and compute the difference in the required probabilities as a function of $q$. 
As a result, we show that the bunkbed conjecture is false for $q \in (0.56, 1.43)$. 

\end{enumerate}

\section{Background}
\label{sec:background}

\subsection{Notation}

{We will use the following notation throughout the paper. For $a,b\in V$, $\nu (ab)$ and 
$\nu (a, b)$ will denote the probability that $a$ and $b$ are in same and different components, respectively. This notation extends to multiple components. For instance, if $a, b, c, d \in V$, $\nu (a, bc, d)$ represents the probability that $a$ is one component, $b$ and $c$ are in a second, and $d$ in a third. Note that we do not place any restriction on the number of components.
}

{For counting we use a slightly different convention. If $a, b \in V$, then by $[a \mid b]$ we mean the number of spanning forests with exactly two components such that $a$ and $b$ are in different components. Similarly, if $a, b, c, d \in V$, $[a\mid bc\mid d]$ represents the number of spanning forests with the minimal number of components (here there will be $3$) such that $a$ is one component $b$ and $c$ are in a second, and $d$ in a third. The expression $[abcd]$ just counts all spanning trees, and so does $[\cdot ]$. We will 
also occasionally use the following extended notation: $[\dots]_1$ will represent the number of forests with exactly one more than the minimal possible number components. For example $[a\mid b]_1$ will have be $2 + 1$ components such that $a$ and $b$ are in different components. 
 }

\subsection{Algebraic graph theory}
\label{sec:alg graph}

The \textit{matrix tree theorem}~\cite[Section II.3, Theorem 12]{bollobas-1998} gives a formula for the number of spanning trees in a graph. Given a graph $G = (V, E)$, recall that the \textit{Laplacian} is defined as the matrix $L_G$ indexed by $V$ whose entries are
\[
L_{G}(v, w) = \begin{cases}
    -1 & \{u,v\}\in E, \\
    \text{deg}(v) & u = v,\\
    0 & \text{otherwise}.
\end{cases}
\]
We will use the notation $L_G(S^c, T^c)$ for $S, T \subset V$ and $|S|=|T|$ to denote the matrix with the rows from $S$ and columns from $T$ removed. Then the matrix tree theorem~\cite[Theorem 4.8]{bapat-2010} says that 
$\det L_G(V \setminus \{v\}, V \setminus \{v\})$ is a positive integer which is equal to the number of spanning trees of $G$ for any $v \in V$. The generalization to counting rooted spanning forests is the \textit{all minors matrix tree theorem}.

\begin{theorem}[{\cite[Page 319]{chaiken-1982}}]
\label{thm:all minors}
Let $G = (V, E)$ and $S, T \subset V$ such that $|S| = |T| = k$
The determinant of $L_G(S^c, T^c)$ is (up to sign) 
the number of forests with $k$ components such that each tree contains exactly one vertex from $S$ and exactly one vertex from $T$. 
\end{theorem}

The special case of $k = 1$ is the usual matrix tree theorem. 

For any real (rectangular) matrix $M$, the \emph{(Moore--Penrose) pseudoinverse} $M^\dagger$ generalizes the notion of the inverse of the matrix~\cite[Problem~7.3.P7]{horn-johnson-1985}. 
We do not need the exact definition of the pseudoinverse, but we need the fact that it is characterized by the following relations. 
\begin{enumerate}
    \item $M M^\dagger M = M$,
    \item $M^\dagger M M^\dagger = M^\dagger$,
    \item $(M M^\dagger)^T = M M^\dagger$, and
    \item $(M^\dagger M)^T = M^\dagger M$.
\end{enumerate}
$M^\dagger$ is known to be unique.

Let $L \equiv L_G$ be the Laplacian of a {connected} graph $G$ with $n$ vertices. 
It is well-known that $L$ is positive semidefinite (PSD).
Then its pseudoinverse satisfies the following properties~\cite{gutman-xiao-2004}.
We will use $I$ (resp. $J$) for the identity matrix (resp. all-ones matrix) of the appropriate size.
\begin{enumerate}
    \item $L^\dagger$ is a real symmetric matrix. Moreover, it is also PSD.
    \item $L J = J L = L^\dagger J = J L^\dagger = 0$,
    \item $L L^\dagger = L^\dagger L = I - J/n$, \label{it: LLdag}
    \item $(L + J/n)^{-1} = L^\dagger + J/n$.    
\end{enumerate}
It is easy to check that either of the last two properties can be used to characterize the pseudoinverse of the Laplacian.

{
There is a natural partial order, denoted $\preceq$, on the set of real symmetric matrices sometimes called the \emph{Loewner order}~\cite{horn-johnson-1985} as follows. If $A, B$ are symmetric matrices, we write $A \preceq B$ if $B-A$ is PSD and $A \prec B$ if $B - A$ is positive definite. In particular, $A$ is PSD (resp. positive definite) if and only if $0 \preceq A$ (resp. $0 \prec A$).
It is a standard result that if $A$ and $B$ are positive definite and $A \prec B$, then
$B^{-1} \prec A^{-1}$~\cite[Corollary 7.7.4]{horn-johnson-1985}.
}

\begin{proposition}
\label{prop:bunkbed pseudoinv}
Let $L$ be the Laplacian of a graph $G$ on $n$ vertices and $\tilde{L}$ be the Laplacian of its bunkbed. Then
\[
\tilde{L}^{\dagger} = \dfrac{1}{2}\left(
\left(
\begin{array}{c|c}
L^{\dagger}&L^{\dagger}\\
\hline
L^{\dagger} & L^{\dagger}
\end{array} \right)
+
\left(
\begin{array}{c|c}
(L+2I)^{-1} & -(L+2I)^{-1}\\
\hline
-(L+2I)^{-1} & (L+2I)^{-1}
\end{array} \right)\right).
\]
\end{proposition}

\begin{proof}
First, note that
\[
\tilde L = \left(
\begin{array}{c|c}
L+I &-I\\
\hline
-I & L+I
\end{array} \right).
\]
Then one checks, using \cref{it: LLdag}, that
\[
\tilde{L} \tilde{L}^\dagger = 
\frac{1}{2}
\left(
\begin{array}{c|c}
  I - J/n   &  I - J/n\\
\hline
  I - J/n   &  I - J/n
\end{array} \right)
+ \frac{1}{2}
\left(
\begin{array}{c|c}
  I &  -I\\
\hline
  -I   &  I
\end{array} \right)
= I - \frac{1}{2n}J.
\]
A similar calculation works for $\tilde{L}^\dagger \tilde{L}$, completing the proof.
\end{proof}

The natural distance on a graph $G = (V, E)$ between two vertices $u$ and $v$ is the length of the shortest path between them. It is well-known that this is a metric. However, there is one more notion of distance on a graph coming from the theory of electrical networks. Suppose each edge in $G$ is replaced by a resistor of unit resistance. Then the \emph{resistance distance} between $u$ and $v$, denoted $R_G(u, v)$, is the effective resistance between them. It turns out that the resistance distance is also a metric (most importantly, it satisfies the triangle inequality) and can be expressed in terms of the pseudoinverse. 
More precisely~\cite{gutman-xiao-2004},
\begin{equation}
\label{resistance distance}
R_G(u, v) = L^{\dagger}_{u, u} + L^{\dagger}_{v, v} - 2L^{\dagger}_{u, v}.
\end{equation}
{If we think of $L^\dagger$ as a linear transformation from $\mathbb{R}^{|V|} \to \mathbb{R}^{|V|}$, and let $\langle \cdot, \cdot \rangle$ be the standard inner product on $\mathbb{R}^{|V|}$, then we can write
\begin{equation}
\label{res dist inner product}
R_G(u, v) = \langle L^\dagger (e_u - e_v), e_u - e_v \rangle.
\end{equation}
}
We remark that if we form the \emph{resistance matrix}, which we also call $R \equiv R_G$, of resistance distances, then we have the equalities~\cite{gutman-xiao-2004},
\[
L R L = -2 L, \quad
\text{and}
\quad
L^\dagger R L^\dagger = -2 (L^\dagger)^3.
\]

From the all minors matrix tree theorem (\cref{thm:all minors}), the determinant of $L_G(\{u, v\}^{c}, \allowbreak \{u, v\}^{c})$, where $u, v$ are vertices in a graph $G$, equals the number of forests with two components with $u, v$ lying in separate components. 
{Recall that $[\cdot]$ is the number of spanning trees.}

\begin{proposition}
\label{prop:pseudo}
Let $L$ be the Laplacian of the graph $G = (V, E)$ and $u, v \in V$.
The determinant of $L(\{u, v\}^{c}, \{u, v\}^{c})$ may be written in terms of the effective resistance $R_G(u, v)$ as
\[
\det L(\{u, v\}^{c}, \{u, v\}^{c}) = [\cdot] \,R_G(u, v).
\]
Thus, $R_{G}(u, v)$ is the number of spanning forests with two trees with $u$ and $v$ in different trees divided by the total number of spanning trees, which we write using the notation from the introduction as
\[
R_G(u, v) = \frac{[u \mid v]}{[\cdot]}.
\]
\end{proposition}

\begin{proof}
Let $L_0$ be the reduced Laplacian obtained by removing the row and column corresponding to the vertex $u$ in $G$. It follows from elementary linear algebra that
\[
(L_0^{-1})_{v,v} = \frac{1}{\det L_0} \det L(\{u, v\}^{c}, \{u, v\}^{c}).
\]
Now, the denominator is $\det L_0 = [\cdot]$. Therefore, it remains to show that 
$(L_0^{-1})_{v,v}$ is $R(u, v)$. This is a standard calculation that is found in \cite[Section 2.4]{ghosh2008minimizing}, for example.
\end{proof}

We will also need the following lemma in \cref{sec:corr}. 

\begin{lemma}
\label{lem:offdiag}
Let $a, b, c, d$ be vertices in a graph $G$ with Laplacial $L$.
With the same notation as above,
\begin{equation*}
\langle L^{\dagger} (e_a - e_b), e_c - e_d\rangle = \dfrac{[ac\mid bd]-[ad\mid bc]}{[\cdot ]}.
\end{equation*}
\end{lemma}

We have been unable to find a reference for this and we include a proof.

\begin{proof}
Since $L^\dagger$ is symmetric, we can use the polarization identity from linear algebra to obtain
\begin{align*}
2\langle L^{\dagger} (e_a - e_b), e_c - e_d\rangle =& \langle L^{\dagger} (e_a - e_d), e_a - e_d\rangle -\langle L^{\dagger} (e_a - e_c), e_a - e_c\rangle\\
&+ \langle L^{\dagger} (e_b - e_c), e_b - e_c\rangle-\langle L^{\dagger} (e_b - e_d), e_b - e_d\rangle.
\end{align*}
Now, using \eqref{res dist inner product} and \cref{prop:pseudo} on each term on the right hand side, 
we have
\begin{align*}
2\langle L^{\dagger} (e_a - e_b), e_c - e_d\rangle 
=& \dfrac{[a\mid d]-[a\mid c]+[b\mid c]-[b\mid d]}{[\cdot]}.
\end{align*}
Considering the block in which $c$ belongs in $[a \mid d]$, and the block in which $d$ belongs in $[a \mid c]$, we get
\[
[a \mid d] = [ac \mid d] + [a \mid cd]  \quad
\text{and} \quad
[a \mid c] = [ad \mid c] + [a \mid cd] ,
\]
and therefore,
\[
[a \mid d] - [a \mid c] = [ac \mid d] - [ad \mid c].
\]
A similar simplification holds for $[b\mid c]-[b\mid d]$. 
Plugging these computations, we obtain
\begin{align*}
2\langle L^{\dagger} (e_a - e_b), e_c - e_d\rangle 
=& \dfrac{[ac\mid d] - [bc \mid d]+ [bd\mid c] - [ad\mid c]}{[\cdot]}.
\end{align*}
Repeating the same kind of calculation for $[ac\mid d] - [bc \mid d]$ and 
$[bd\mid c] - [ad\mid c]$, we finally get
\[
2\langle L^{\dagger} (e_a - e_b), e_c - e_d\rangle = 2\,\,\dfrac{[ac\mid bd]-[ad\mid bc]}{[\cdot]},
\]
completing the proof.
\end{proof}

\subsection{Alternate bunkbed model}
\label{sec:alt bunkbed}

A few years ago, the second author~\cite{linusson_2011,linusson_2019} investigated the bunkbed conjecture for outerplanar graphs.\footnote{Although the paper is withdrawn, some constructions in the paper continue to be useful.}
Recall that a graph is \emph{outerplanar} if it can be drawn in the plane with non-crossing edges and with all the vertices lying on a circle. This is a minor-closed family of graphs and \cite{linusson_2011} studied what a minimal counterexample (minimal in terms of the number of edges) ought to look like. 

Key to this analysis are two new models, one a generalization of the bunkbed conjecture and the other a reduction. The first refines which vertical edges appear in the bunkbed graph. Given a subset $T \subset V(G)$ which we call \emph{posts}, we look at the graph $\tG_T$, which is a subgraph of the regular bunkbed graph $\tG$, where the vertical edges corresponding to vertices in $T$ are present and the other vertical edges are absent. One then performs percolation with parameter $p$ on the two sets of horizontal edges in $\tG_T$. 
The vertical edges at the posts are always open.
One {has the following question:}

\begin{problem}
\label{ques:postsRC}
Let $G$ be a graph, $T \subset V$ a set of posts, $u,v\in V$ and $\tG_T$ the associated bunkbed graph. 
Let $\rc{\tG_T}$
be the random cluster measure on $\tG_T$ with parameters $p, q$, conditioned on the vertical edges corresponding to the posts being present. Then for 
{which $G,T$} and {which values of $p, q$} is it true that 
\begin{align} 
\label{bunkbed RC posts}
\rc{\tG_T}(u_1 \leftrightarrow v_1) \geq 
\rc{\tG_T}(u_1 \leftrightarrow v_2)?
\end{align}
\end{problem}

{Setting $q = 1$, we get a refinement of \cref{thm:bunkbed false} for percolation as well. The counterexample given in \cite{gladkov_pak_zimin_2024} was in fact for a specific $\tG_T$ with $|T|=3$. 
It is easy to see (using a simple conditioning argument) that if for a given $G$,
\cref{ques:postsRC} is true for all $T$, then 
\cref{prob:origRC} is true as well for that $G$. }

In \cite{linusson_2011} was also presented a different model for the bunkbed conjecture, which we call 
the \emph{alternate bunkbed conjecture on $(G, T)$},
where $G$ is the graph and $T$ the subset of posts as above. 
{Let $\altperc{G}T$ be the uniform measure on all red-blue colourings of edges.
An \textit{admissible path} between vertices in $G$ is a path that only change colours at the posts. 
We now consider admissible paths in $G$ with respect to $\altperc{G}T$.
We use the notation $u \leftrightarrow_{RR} v$ (resp. $u \leftrightarrow_{RB} v$) to mean 
that there is 
an admissible path 
from $u$ to $v$ that starts with a red edge and ends with a red edge (resp. a blue edge).}

{We know that the bunkbed conjecture is false and therefore also the alternate version. But we can ask whether the alternate bunkbed inequality holds for some families of graphs.}

\begin{problem}
\label{ques:alt bunkbed}
Let $G$ be a graph,  $T \subset V$ a set of posts and $u, v$ vertices in $G$. 
Then for what $G$ does the inequality
\begin{align} 
\label{alt bunkbed}
\altperc{G}T(u \leftrightarrow_{RR} v) \geq \altperc{G}T(u \leftrightarrow_{RB} v),
\end{align}
hold?
\end{problem}

We may similarly look at the alternate model for the random cluster measure. 
In this model, given a red-blue colouring, the number of connected components is computed using admissible paths. 
Here, we focus on the arboreal gas weak limit, which we denote by $\altforest{G}{T}$, where $G$ is a graph and $T$ is a set of posts. 
In this measure, there is a restriction on the red-blue colourings with respect to cycles, namely that any colouring of any cycle must not be admissible.
This is because otherwise it would correspond to a cycle in $\tG_T$, which we do not have in the arboreal gas measure.
We have the following conjecture.

\begin{conjecture}
\label{conj:alt forest}
Let $G$ be a graph,  $T \subset V$ a set of posts and $u, v$ vertices. 
Then
\begin{align} 
\label{alt bunkbed forests}
\altforest{G}{T}(u \leftrightarrow_{RR} v) \geq \altforest{G}{T}(u \leftrightarrow_{RB} v).
\end{align}
\end{conjecture}

If true, this would imply \cref{conj:forest}.

The proof strategy in \cite{linusson_2011} was to show that a minimal counterexample does not exist. In that context the following result was proved in the context of percolation\footnote{The error in \cite{linusson_2011} was in another argument.}
However, the reductions (which use assorted deletions and contractions), work well for all random cluster measures, and in particular, for the arboreal gas measure as well.

\begin{theorem}[\cite{linusson_2011}]
\label{thm:linusson}
The following are equivalent. 
\begin{enumerate}
\item There is an outerplanar graph $G$, a vertex subset $T$ and vertices $u, v$ such that the alternate bunkbed conjecture for the arboreal gas measure, \cref{conj:alt forest}, is false. 

\item There is an outerplanar graph $G$, a vertex subset $T$, and vertices $u, v$ such that the following hold:
\begin{enumerate}
   \item $(G, T)$ is a minimal (in the number of edges) counterexample to \cref{conj:alt forest}. 
    \item There are no non-post vertices of degree $2$.
    \item There are no adjacent post vertices. 
    \item There are two non-intersecting boundary paths from $u$ to $v$.
   \item Every chord separates $u$ and $v$. In particular, $u$ and $v$ have degree $2$. 
    \item Neither $u$ nor $v$ are in $T$.
    \item The vertices in $T$ do not form a cutset. 
\end{enumerate}
\end{enumerate}
\end{theorem}

\begin{proof}
{The proof is the same as in \cite{linusson_2011}, so we just explain (b) and let the interested reader look up the remaining statements there. Assume $x$ is a vertex in $G$ with degree 2 and $x\notin T$. Then we can condition on the two possibilities; either the two edges at $x$ have different colours or they have the same colour. The first case is equivalent to removing $x$, since $x$ cannot be used. The second to contracting one of the edges incident to $x$; this could never create or remove a cycle and is thus true also for the arboreal gas measure. Since both sides in the inequality \eqref{alt bunkbed forests} are the same linear combination of the corresponding probabilities in these two cases, $G,T$ cannot be a minimal counterexample. This proves (b).}
\end{proof}

We will use this alternate bunkbed model for forests in \cref{sec:outerplanar}
to prove new results for outerplanar graphs. 

\subsection{Basics of correlation inequalities}

Correlation inequalities are a key tool in statistical physics to study the large scale behaviour of spin and other models and constitute an important subarea of their own. We will only be concerned with the most basic of these inequalities, the Harris--Kleitman inequality and its consequences. Recall that an event (subset) $X \subset 
\{0, 1\}^n$ is said to be \emph{increasing} if it is upward closed. 

\begin{theorem}[Harris--Kleitman inequality]
\label{thm:harris ineq}
Let $\mu$ be a product measure on $\{0, 1\}^n$ and let $X, Y$ increasing events. Then 
\[
\mu(X \cap Y) \geq \mu(X)\, \mu(Y).
\]
In other words, any two increasing events are positive correlated. 
\end{theorem}
The well known \textit{FKG inequality} implies that increasing events are also positively correlated with respect to random cluster measures $\rc{G}$
for $q \geq 1$.
This means in particular that with respect to a percolation measure or more generally a random cluster measure with $q \geq 1$,
any two connectivity events $a \leftrightarrow b$ and $b \leftrightarrow c$ are positively correlated, where $a, b, c,$ are vertices. 

In recent work, Gladkov proved a beautiful strengthening of this. 

\begin{theorem}[{\cite[Theorem 3.2]{gladkov_2024}}]
\label{thm:gladkov}
Let $q \geq 1$ and let $\mu \equiv \mu^{RC}_{G, p, q}$ be a random cluster measure on a graph $G$ and let $a, b, c$ be vertices. Then 
\[
\mu(abc)\, \mu(a, b, c) \geq e_2(\mu(ab, c), \mu(ac, b), \mu(bc, a)),
\]
where $e_2(x, y, z) = xy + xz + yz$ is the elementary symmetric polynomial of degree $2$.
\end{theorem}

Proving correlation inequalities for random cluster measures in the regime $q < 1$ is a challenging task; see \cite{pemantle2000towards} for an extended discussion on why this is the case (and why nevertheless, it is worth seeking for). We believe that \cref{thm:gladkov} holds for $q < 1$ random cluster measures as well. We take a small step towards this in this paper. We postpone a discussion to \cref{sec:corr}, but recall here an important result that we will use. 

An important inequality for the spanning tree measure is the well known \emph{Rayleigh monotonicity principle}. Recall the spanning tree measure $\ust{G}$ defined earlier. 
We let $\ust{G}(e)$ 
denote the probability that an edge $e$ belongs to a uniformly chosen spanning tree.

\begin{theorem}[{\cite[Equation (3)]{choe-2008}, \cite[Section 1.4.1]{doylesnell}}]
\label{thm:rayleigh}
Let $G = (V, E)$ be a graph and $e, f$ edges in $G$. By $G \setminus f$, we mean the graph $(V, E \setminus \{f\})$, which we assume to be connected. Then we have
\[
\ust{G}(e) \leq \ust{G \setminus f}(e).
\]
More generally, for any two vertices $a, b$ in $G$, we have
\[
\frac{[a\mid b]^G}{[\cdot]^G} \leq \frac{[a\mid b]^{G \setminus f}}{{[\cdot]^{G \setminus f}}}.
\]
Note that this implies that for any connected vertex subgraph $H$ of $G$, we have 
\[
\frac{[a\mid b]^G}{[\cdot]^G} \leq \frac{[a\mid b]^{H}}{{[\cdot]^{H}}}.
\]
\end{theorem}

We will see a strengthening of Rayleigh's Theorem in \cref{lem:strong rayleigh}.

\section{Random cluster model in the $p \uparrow 1$ limit}
\label{sec:pto1}

While the bunkbed problem for percolation is false in general, there are a few notable positive results. 
Hutchcroft, Kent and Nizi\'c-Nikolac~\cite{hutchcroft_et_al_2023} have shown that this bunkbed conjecture holds in the limit $p \uparrow 1$, a result that was reproved using a different method by Hollom~\cite{hollom_2024}. 
We extend this result to the random cluster measure. 

\begin{theorem}\label{thm:p1}
Let $G$ be a graph, $T \subset V(G)$ a collection of posts, $u, v$ vertices in $G$, and $\rc{\tG_T}$ the random cluster measure on the bunkbed graph $\tG_T$. Then, for fixed $q$, as $p \uparrow 1$,
\[
\rc{\tG_T}(u_1 \leftrightarrow v_1) \geq \rc{\tG_T}(u_1 \leftrightarrow v_2).
\]
\end{theorem}

We now show that the proof in \cite{hutchcroft_et_al_2023} can be adapted to our setting. 

\begin{proof}
To avoid unnecessary notation, we will write $\tG$ instead of $\tG_T$ throughout the proof.
For any tripartition of edges in $E(G)$, which we denote by $S = (S_0, S_1, S_2)$, we may look at those configurations of edges in the bunkbed graph $\tG$ such that 
\begin{itemize} 
\item For every edge $e$ in $S_0$, neither $e_1$ nor $e_2$ appears,
\item  For every edge $e$ in $S_1$, exactly one of $e_1$ and $e_2$ appears, and
\item  For every edge $e$ in $S_2$, both $e_1$ and $e_2$ appear.
\end{itemize}

Let $\mathcal S$ be the set of all such tripartitions. Let further, for $S \in \mathcal{S}$, $\pi(S)$ be the graph obtained from $G$ by deleting all the edges in $S_0$ and contracting all edges in $S_2$. 
When an edge $(x, y)$ is contracted and the two vertices $x$ and $y$ are merged, the resulting vertex will be a post in $\pi(T)$ if at least one of $x,y\in T$. If multiple edges between the same pair of vertices arise upon contraction we remove extra edges so there is only one edge left.

Additionally, we say that a path is \emph{post-free} if no vertex on the path is a post. Let $d(S)$ be the length of the shortest post-free path between $u$ and $v$ in $\pi(S)$. Define $\mathcal{S}_{\infty} \subset \mathcal{S}$ to be those tripartitions where this length is infinity (i.e. no such path exists, for example if $u$ or $v$ have become a post by contractions). 
Similarly, define $\mathcal{S}_{\leq 1} \subset \mathcal{S}$ to be those tripartitions $S$ such that $d(S) \leq 1$, and similarly $\mathcal{S}_{\geq 2} \subset \mathcal{S}$ to be those where $2 \leq d(S) < \infty$. 

We need to show that 
\[
\rc{\tG}(u_1 \leftrightarrow v_1) - \rc{\tG}(u_1 \leftrightarrow v_2)\ge 0.
\]
First we rewrite the difference as \[
\rc{\tG}(u_1 \leftrightarrow v_1) - \rc{\tG}(u_1 \leftrightarrow v_2) = 
\sum_{S\in \mathcal S} F(S)\rc{G}(S),
\]
where $\rc{\tG}(S)$ is the probability of $S$ being the tripartition, and $F(S)$ is the difference above conditioned on the tripartition being $S$.

We will write probabilities with respect to the random cluster measure as 
$\rc{\tG}$
and the percolation measure as 
$\perc{\tG}$. 
The argument by Hutchcroft, Kent and Nizi\'c-Nikolac~\cite{hutchcroft_et_al_2023} is the following: 
\begin{enumerate}
    \item In \cite[Lemma 1]{hutchcroft_et_al_2023}, it is proved that for $S \in \mathcal{S}_{\infty}$, one has $F(S) = 0$. 
    This uses a mirroring argument from \cite{linusson_2011}, which is valid for any random cluster measure.

    \item \label{it:lem3} 
    In \cite[Lemma 3]{hutchcroft_et_al_2023}, it is proved that configurations in $\mathcal{S}_{\geq 2}$ are rare when $p$ is close to 1. 
    Indeed one has that 
    \[
    \perc{\tG}(\mathcal{S}_{\geq 2} )\leq 2e\left(\dfrac{1-p}{p}\right)^2 \perc{\tG}(\mathcal{S}_{\leq 1}),
    \]
    where $e$ is the base of the natural logarithm.

    \item \label{it:lem2} 
    In \cite[Lemma 2]{hutchcroft_et_al_2023}, it is proved that for $S \in \mathcal{S}_{\leq 1}$, we have that $F(S) \geq 2^{1-|E(G)|}$. 
\end{enumerate}

We will now use the inequalities in \cref{it:lem3,it:lem2}.
Let $n$ be the number of vertices in $\tG$.
We also need the simple inequalities that for any event $X$, 
\begin{align}\label{ineq:q}
q^{-n}\perc{\tG}(X) \leq \rc{\tG}(X) \leq q^n \perc{\tG}(X),
\end{align}
for $q > 1$ and 
\begin{align}\label{ineq:q2}
q^{n} \perc{\tG}(X) \leq \rc{\tG}(X) \leq q^{-n}\perc{\tG}(X),
\end{align}
for $q < 1$. In what follows, we assume that $q < 1$; the other case is similar. We use the above inequality to see that 
\[
\rc{\tG}(\mathcal{S}_{\leq 1}) \geq q^{n} \perc{\tG}(\mathcal{S}_{\leq 1}),
\qquad 
\rc{\tG}(\mathcal{S}_{\geq 2} )
\leq q^{-n} \perc{\tG}(\mathcal{S}_{\geq 2} )
\leq q^{-n} \,\bigg(2e\left(\dfrac{1-p}{p}\right)^2\bigg) \perc{\tG}(\mathcal{S}_{\leq 1}).
\]
We then use $F(S)\ge -1$ for all $S\in \mathcal{S}_{\geq 2}$ to get
\[
\sum_{S\in \mathcal S} F(S)\rc{\tG}(S) \geq \left(q^n2^{1-|E(G)|}  - 2eq^{-n}\left(\dfrac{1-p}{p}\right)^2 \right) \perc{\tG}(\mathcal{S}_{\leq 1}).
\]
To complete the proof, we observe that this is positive when $p$ is close enough to $1$, say when
\[
\left(\dfrac{1-p}{p}\right)^2 \leq 2^{-|E(G)|-2}q^{2n},
\]
and so $p \geq 1/(1 + 2^{-|E(G)|/2-1}q^{n})$, completing the proof.
\end{proof}

\section{Complete graphs and complete bipartite graphs}
\label{sec:complete and bipartite}

We will prove the bunkbed conjecture for the random cluster measure in these two special cases.

\subsection{Complete graphs}
\label{sec:complete}

\begin{theorem}
\label{thm:complete}
The bunkbed conjecture holds for the random cluster measure on the complete graph $K_n$. 
\end{theorem}

The proof is a simple extension of the proof in the regular percolation case  given by Hintum and Lammers~\cite{van_Hintum_2019}  and we proceed by describing their short and elegant proof. 
Consider a configuration of edges in the subgraph of the bunkbed graph of $K_n$ after removing the two ending vertices $v_1, v_2$. Hintum and Lammers show that, conditioned on any given configuration, the bunkbed conjecture is true when the starting vertex 
$u_1$ is chosen to be a uniformly random vertex in $V \setminus \{v_1\}$.

\begin{proof}
Let $\tG$ be the bunkbed graph of $K_n$ and let $v_1, v_2$ be paired vertices in $\tG$ with $v_1$ below, as usual. Fix a configuration of edges in the subgraph $V(\tG) \setminus \{v_1, v_2\}$ without worrying about the probability of realising this configuration for now. Let $\mathcal{C} = \{C_1, \ldots, C_m\}$ be the connected components in this subgraph.
We will use $\ell(C_j)$ and $u(C_j)$ to denote the number of vertices of $C_j$ in the lower graph and upper graph respectively, and $v\sim C$ to denote that there is an edge from $v$ to some vertex in the component $C$.
In the calculation of probabilities, we use the standard notation $\wedge$ for the joint occurrence of events.

Conditioned on $\mathcal{C}$, the probability that $u_1$ (which is a uniformly random vertex downstairs, distinct from $v_1$) is connected to $v_1$ but not to $v_2$ is given by 
\[
\sum_{j = 1}^m 
\rc{\tG}\left[(u_1 \in C_j) \wedge (v_1 \sim C_j) \wedge (v_2 \not\sim C_j) \wedge (\not\!\exists k \neq j \mid v_1 \sim C_k \sim v_2)\right].
\]
Also, the probability that $u_1$ is connected to $v_2$ but not to $v_1$ is given by 
\[
\sum_{j = 1}^m  \rc{\tG}\left[(u_1 \in C_j) \wedge (v_2 \sim C_j) \wedge (v_1 \not\sim C_j) \wedge (\not\!\exists k \neq j \mid v_1 \sim C_k \sim v_2)\right].
\]
Since $u_1$ is chosen uniformly, we may rewrite the first expression as 
\[
\sum_{j = 1}^m \dfrac{\ell(C_j)}{n-1} \rc{\tG}\left[ (v_1 \sim C_j) \wedge (v_2 \not\sim C_j) \wedge (\not\exists k \neq j \mid v_1 \sim C_k \sim v_2)\right].
\]
The other expression may be similarly simplified. 

We need to show that 
\begin{multline}
\label{lhs geq rhs}
\sum_{j = 1}^m \dfrac{\ell(C_j)}{n-1} \Bigg( \rc{\tG}\left[(v_1 \sim C_j) \wedge (v_2 \not\sim C_j) \wedge (\not\!\exists k \neq j \mid v_1 \sim C_k \sim v_2)\right]\Bigg)\\
\geq \sum_{j \in [m]}\dfrac{\ell(C_j)}{n-1}  \Bigg(\rc{\tG}\left[(v_2 \sim C_j) \wedge (v_1 \not\sim C_j)\wedge (\not\!\exists k \neq j \mid v_1 \sim C_k \sim v_2)\right]\Bigg).
\end{multline}
In the case of percolation, we may write out explicit expressions for the quantities above. However, in the case of the more general random cluster measures, we need to do some additional bookkeeping. We now examine how $v_1, v_2$ are connected to this fixed configuration. 

Let $r = 1-p$ for convenience. 
Given the connected components $C_1, \ldots, C_m$ in $\tG \setminus \{v_1, v_2\}$, we naturally obtain a partition 
{$\mathcal{P} = (A_1, A_2, A_3)$ of $[m]$ into three parts,}
where 
\begin{itemize}
\item $A_1$ consists of components to which $v_1$ is connected but $v_2$ is not.
\item $A_2$ consists of components to which $v_2$ is connected but $v_1$ is not.
\item $A_3$ consists of components to which neither $v_1$ nor $v_2$ is connected.
\end{itemize}
{The fourth possibility, namely that both $v_1$ and $v_2$ are connected to some $C_i$, does not contribute to either side of \eqref{lhs geq rhs}; hence we may assume that there is no such component.} 

All the vertices in the components $C_i$ for $i \in A_1$ together with $v_1$ form a single connected component, as do all the vertices in $C_i$ for $i \in A_2$, together with $v_2$. The vertices in $C_i$ for any $i \in A_3$ form separate components, leading to a total of $|A_3|+2$ components. 

We multiply out both quantities in \eqref{lhs geq rhs} by $n-1$ to simplify calculations. 
Summing over all partitions $(X, Y, Z)$ of $[m] \setminus \{j\}$ into three parts, we have that the left hand side is given by 
\begin{align*} 
L 
=\sum_{j = 1}^m \ell(C_j) & \sum_{X \amalg Y \amalg Z = [m]\setminus \{j\}} \Bigg( \rc{\tG}\left[(v_1 \sim C_j) \wedge (v_2 \not\sim C_j) \right. \\
& \left. \wedge (\not\exists k \neq j \mid v_1 \sim C_k \sim v_2)\wedge \mathcal{P} = (X \cup\{j\}, Y, Z)\right]\Bigg)\\
=\sum_{j = 1}^m \ell(C_j) & \,\sum_{X \amalg Y \amalg Z = [m]\setminus \{j\}} q^{|Z|+2} (1-r^{\ell(C_j)})r^{u(C_j)} \prod_{i \in X} (1-r^{\ell(C_i)}) r^{u(C_i)} \\
& \times\prod_{i \in Y}  (1-r^{u(C_i)}) r^{\ell(C_i)}\prod_{i \in Z}  r^{\ell(C_i)} r^{u(C_i)}\\
=\sum_{j = 1}^m \ell(C_j) & (1-r^{\ell(C_j)}) r^{u(C_j)}\sum_{X \amalg Y \amalg Z = [m]\setminus \{j\}} f_j(\mathcal{C}, X, Y, Z),
\end{align*}
where the quantity $f_j(\mathcal{C}, X, Y, Z)$ depends on $X, Y, Z$ and $\mathcal{C}$ as above. 
The important property here is that for a fixed $Z$ the random cluster measure behaves as percolation. Similarly, the right hand side of \eqref{lhs geq rhs} becomes
\begin{align*} 
R 
=\sum_{j = 1}^m \ell(C_j) &\, \sum_{X \amalg Y \amalg Z = [m]\setminus \{j\}} \Bigg( \rc{\tG}\left[(v_2 \sim C_j) \wedge (v_1 \not\sim C_j) \right. \\
& \left. \wedge (\not\exists k \neq j \mid v_1 \sim C_k \sim v_2)\wedge \mathcal{P} = (X, Y\cup\{j\}, Z)\right]\Bigg)\\
=\sum_{j = 1}^m \ell(C_j) & \,(1-r^{u(C_j)}) r^{\ell(C_j)}\sum_{X \amalg Y \amalg Z = [m]\setminus \{j\}} f_j(\mathcal{C}, Y,X, Z),
\end{align*}
where the function $f_j$ is the same as before. Note how changing the roles of $v_1, v_2$ swaps $X$ and $Y$. 
Since the inner sum in $R$ is summing over all partitions $X \amalg Y \amalg Z$, it will be equal to the inner sum in $L$.

We need to show that $L - R \geq 0$, i.e. that
\begin{align*}
0 \leq    &\sum_{j = 1}^m  \ell(C_j) \Bigg((1-r^{\ell(C_j)}) r^{u(C_j)} - (1-r^{u(C_j)}) r^{\ell(C_j)}\Bigg) \sum_{X \amalg Y \amalg Z = [m]\setminus \{j\}} f_j(\mathcal{C}, X, Y, Z)\\
    =&\sum_{j \in [m]} \ell(C_j) \Bigg( r^{u(C_j)} - r^{\ell(C_j)}\Bigg)\sum_{X \amalg Y \amalg Z = [m]\setminus \{j\}} f_j(\mathcal{C}, X, Y, Z).
\end{align*}
Intuitively this is non-negative because the positive terms have $\ell(C_j)>u(C_j)$. To prove this rigorously, we now switch the roles of $v_1$ and $v_2$, and simultaneously $u_1$ and $u_2$. 
By symmetry, the above value also equals
\[
\sum_{j = 1}^m  u(C_j) \Bigg( r^{\ell(C_j)} - r^{u(C_j)}\Bigg)\sum_{X \amalg Y \amalg Z = [m]\setminus \{j\}}f_j(\mathcal{C}, X, Y, Z).
\]
Summing the two, we get
\[
\sum_{j = 1}^m  \bigg(\ell(C_j)-u(C_j)\bigg) \Bigg( r^{u(C_j)} - r^{\ell(C_j)}\Bigg)\sum_{X \amalg Y \amalg Z = [m]\setminus \{j\}} f_j(\mathcal{C}, X, Y, Z).
\]
Each term is positive since $(x-y)(r^y-r^x) \geq 0$ whenever $r \in [0, 1]$ and $x, y \in \mathbb{R}$.
\end{proof}

\subsection{Complete bipartite graphs}
The bunkbed conjecture for complete bipartite graphs was proved by Richthammer, using a extension of the argument in \cite{van_Hintum_2019}. 
He observed that the argument in \cref{sec:complete}, in verbatim fashion, proves in~\cite[Corollary 1]{richthammer_2022} that 
\[
\rc{\tG}(u_1 \leftrightarrow v_1) \geq \rc{\tG}(u_1 \leftrightarrow v_2),
\]
provided that $u$ and $v$ are in different sides of the bipartition. Indeed, the argument actually shows the following slightly more general statement. Let $G$ be a graph and $v$ be a vertex so that given any two neighbours of $v$, there is an automorphism of $G$ that fixes $v$ and swaps these neighbours. Then for any neighbour $u$ of $v$, the bunkbed conjecture is true for $u$ and $v$. In particular, this shows that the conjecture is true for complete multipartite graphs, provided $u, v$ are not in the same partition. The proof of \cref{thm:complete} is easily adapted to this setting and we omit the details.

We focus on complete bipartite graphs here.
Let $G = K_{s, t}$ be a complete bipartite graph. 

\begin{theorem}
\label{thm:complete bipartite}
The bunkbed conjecture holds for the random cluster measure on complete bipartite graphs for any pair of vertices $u, v$. 
\end{theorem}

\begin{proof} 
As mentioned above, the proof in \cite[Corollary 1]{richthammer_2022} covers the situation when $u,v$ are in different sides of the bipartition. 
So now suppose $u, v$ are on the same side of the bipartition.
Without loss of generality, assume they are on the left side.
Observe that relevant expression for the bunkbed conjecture can be written as
\begin{align*}
    I =& \,\rc{\tG}(u_1 \leftrightarrow v_1) - \rc{\tG}(u_1 \leftrightarrow v_2)\\
    =& \,\rc{\tG}(u_1 v_1 , u_2 , v_2) + \rc{\tG}(u_1 v_1 , u_2 v_2) 
    + \rc{\tG}(u_1 u_2 v_1, v_2) + \rc{\tG}(u_1 v_1 v_2, u_2)\\
 &- \, \rc{\tG}(u_1 v_2 , u_2 , v_1) - \rc{\tG}(u_1 v_2 , u_2 v_1) 
 - \rc{\tG}(u_1 u_2 v_2, v_1) -\rc{\tG}(u_2 v_1 v_2, u_1).
\end{align*}
We can simplify this by noting that 
\[
\rc{\tG}(u_1 u_2 v_1, v_2) = \rc{\tG}(u_1 u_2 v_2, v_1).
\]
This is because the map that switches $w_1$ and $w_2$ for each $w \in V(G)$ does not change connection probabilities. Similarly, we have that  
\[
\rc{\tG}(u_1 v_1 v_2, u_2) = \rc{\tG}(u_2 v_1 v_2, u_1).
\]
We may may therefore write 
\begin{align*}
  I = \rc{\tG}(u_1 v_1 , u_2 , v_2) + \rc{\tG}(u_1 v_1 , u_2 v_2) 
    - \rc{\tG}(u_1 v_2 , u_2 , v_1) - \rc{\tG}(u_1 v_2 , u_2 v_1).
\end{align*}
By switching the roles of $u_1, u_2$ and $v_1, v_2$, this also equals 
\begin{align*}
    I =& \,\rc{\tG}(u_2 \leftrightarrow v_2) - \rc{\tG}(u_2 \leftrightarrow v_1)\\
    =& \,\rc{\tG}(u_2 v_2 , u_1 , v_1) + \rc{\tG}(u_2 v_2 , u_1 v_1) 
    -\rc{\tG}(u_2 v_1 , u_1 , v_2) - \rc{\tG}(u_2 v_1 , u_1 v_2).
\end{align*}
Summing, we get
\begin{align*}
    2I =& \left(\rc{\tG}(u_2 v_2 , u_1 , v_1) + \rc{\tG}(u_1 v_1 , u_2 , v_2)  
    -\rc{\tG}(u_2 v_1 , u_1 , v_2) \right. \\
    & \left. - \rc{\tG}(u_1 v_2 , u_2 , v_1)\right)
    +2 \left(\rc{\tG}(u_2 v_2 , u_1 v_1) -   \rc{\tG}(u_2 v_1 , u_1 v_2)\right).
\end{align*}
We will now show that both expressions, 
\[
M = \rc{\tG}(u_2 v_2 , u_1 , v_1) + \rc{\tG}(u_1 v_1 , u_2 , v_2)  
- \rc{\tG}(u_2 v_1 , u_1 , v_2) - \rc{\tG}(u_1 v_2 , u_2 , v_1),
\]  
as well as 
\[
N = \rc{\tG}(u_2 v_2 , u_1 v_1) -   \rc{\tG}(u_2 v_1 , u_1 v_2),
\] 
are positive. 

Let $\mathcal{C} = \{C_1, \ldots, C_m\}$ be a partition of $\tG\setminus \{u_1, u_2, v_1, v_2\}$ into connected components. 
We will show that in fact $M$ and $N$, conditioned on $\mathcal{C}$, are positive. 
Consider four subsets $U_1, V_1,U_2, V_2$ of $[m]$
such that $u_1, v_1,u_2, v_2$ are connected by an edge to the components of $\mathcal{C}$ indexed by $U_1, V_1, U_2, V_2$ respectively. 
Let $\mathcal{P} = (U_1, V_1,U_2, V_2)$ be the tuple of these subsets.
Not all collections of subsets are relevant to the calculation of the expressions $M$ and $N$. For instance, 
$\rc{\tG}(u_1 v_1 , u_2 v_2)$ is nonzero only when $|U_1 \cap V_1| > 0$, $|U_2 \cap V_2| > 0$ and the sets $U_1 \cup V_1$ and $U_2 \cup V_2$ are disjoint. Also note that no component can have edges to three or four of these vertices.

It will be convenient to work with partitions $A \amalg B \amalg C \amalg D \amalg X \amalg Y = [m]$, where $X$ (resp. $Y$) indexes components which are connected to more than one (resp. none) of $u_1, u_2, v_1, v_2$. 
We will always assume that $|X| > 0$, i.e. we will only be interested in configurations where at least two of $u_1, v_1, u_2, v_2$ are in the same component. This is because any expression occurring in $M$ and $N$ is of this sort. We see that $M$ conditioned on $\mathcal{C}$ is 
\begin{align*}
    & \sum_{A \amalg B \amalg C \amalg D \amalg X \amalg Y = [m]} 
    \rc{\tG}((u_2 v_2 , u_1 , v_1) \wedge \mathcal{P} = (A, B, C \cup X, D \cup X))\\
    +&\sum_{A \amalg B \amalg C \amalg D \amalg X \amalg Y = [m]} 
    \rc{\tG}((u_1 v_1 , u_2 , v_2) \wedge \mathcal{P} = (A \cup X, B \cup X, C, D))\\
    -&\sum_{A \amalg B \amalg C \amalg D \amalg X \amalg Y = [m]} 
    \rc{\tG}((u_2 v_1 , u_1 , v_2) \wedge \mathcal{P} = (A, B \cup X, C \cup X, D))\\
    -&\sum_{A \amalg B \amalg C \amalg D \amalg X \amalg Y = [m]} 
    \rc{\tG}((u_1 v_2 , u_2 , v_1) \wedge \mathcal{P} = (A \cup X, B, C, D \cup X)).
\end{align*}
Note that in all cases, the number of connected components equals $|Y|+3$. Again, let 
$r=1-p$.
Further, given the subset $C_i$, this might contain vertices from both the left and right sides of the bipartition as well as from the lower and upper sets of vertices with respect to the bunkbed structure. We will use $\ell R(C_i)$ to denote the number of vertices in $C_i$ that are in the lower part of the right partition, $uR(C_i)$ to denote the number of vertices in the upper part of the right partition. 

Recall that $u,v$ are in the left partition of $K_{s,t}$.
We have that 
\[
\rc{\tG}((u_2, v_2 , u_1 , v_1) \wedge (\mathcal{P} = (A, B, C \cup X, D \cup X))
\] 
is proportional to 
\begin{align*}
\prod_{i \in A}\left(1-r^{\ell R(C_i)}\right)\prod_{i \in B}\left(1-r^{\ell R(C_i)}\right)\prod_{i \in C}\left(1-r^{uR(C_i)}\right)\prod_{i \in D}\left(1-r^{uR(C_i)}\right)\prod_{i \in X} \left(1-r^{uR(C_i)}\right)^2q^{|Y|+3}
\end{align*}
Assume now that $A, B, C, D$ are fixed. Then we have an inner sum over $X$ which will be
\[
\prod_{i \in X}\left(1-r^{uR(C_i)}\right)^2 q^{|Y|+3}.
\]
We may calculate the other three terms similarly and by fixing $A,B,C,D$ we end up getting a joint inner sum over $X$
\begin{multline*}
\sum_{X}\left(\prod_{i \in X}\left(1-r^{uR(C_i)}\right)^2 + \prod_{i \in X}\left(1-r^{\ell R(C_i)}\right)^2 - 2\prod_{i \in X}\left(1-r^{uR(C_i)}\right)\left(1-r^{\ell R(C_i)}\right)\right) q^{|Y|+3}\\ 
=\sum_{X\subset [m]\setminus \{A\cup B\cup C \cup D\}}\left (\prod_{i \in X}\left(1-r^{uR(C_i)}\right) -\prod_{i \in X}\left(1-r^{\ell R(C_i)}\right)\right )^2  q^{|Y|+3} \geq 0.
\end{multline*}

For the case of the expression $N$, we again work with partitions, but of a slightly different sort. We take a partition 
\[
A \amalg B \amalg C \amalg D \amalg X \amalg Y \amalg Z= [m],
\]
and look at the subset collections $(A \cup X, B \cup X, C \cup Y, D \cup Y)$ and $(A \cup X, B \cup Y, C \cup Y, D \cup X)$. Note that in both cases, the number of connected components is $|Z| + 2$. 
\begin{align*}
\sum_{X,Y}\left(\prod_{i \in X}\left(1-r^{\ell R(C_i)}\right)^2\prod_{i \in Y}\left(1-r^{uR(C_i)}\right)^2 - \prod_{i \in X\cup Y}\left(1-r^{\ell R(C_i)}\right)\left(1-r^{uR(C_i)}\right)\right) q^{|Z|+2}.
\end{align*}
Pairing up each term over $X,Y$ with the term over $Y,X$, i.e. the sets are swapped, we get a simple quadratic inequality exactly as above. 
\end{proof}

\section{Outerplanar graphs and the bunkbed conjecture for forests}
\label{sec:outerplanar}

As discussed in \cref{sec:alt bunkbed}, it is sufficient to look at the alternate bunkbed conjecture for forests, \cref{conj:alt forest}.
Let $G = (V, E)$ be a minimal (in the number of edges) counterexample for the alternate bunkbed conjecture for forests on outerplanar graphs with $T \subset V$ being the set of posts. 
Recall that we use $\altforest{G}{T}$ for the uniform measure on red-blue colourings of edges of $G$, where edges on non-post vertices have the same colour and there are no cycles of the same colour.
Because of the last condition, there is a further salubrious consequence. 

\begin{figure}[ht!]
\centering 
\begin{tikzpicture}
  \node[circle, fill=black, label=above:$u$, inner sep=2pt] (u) at (0,0) {};
      \node[circle, fill=black, label=right:$v$, inner sep=2pt] (v) at (6.05,0) {};
  \node[circle, fill=black, label=left:$x$, inner sep=2pt] (x) at (2,1.73) {};
  \node[circle, fill=black, label=left:$y$, inner sep=2pt] (y) at (2,-1.73) {};
  \node[circle, fill=black, label=above:$z$, inner sep=2pt] (z) at (5,1.73) {};
  \node[circle, fill=black, label=right:$ $, inner sep=0pt] (b) at (5,-1.73) {};
\draw (2,1.73) circle (6pt);
\draw (u) -- (x) -- (z) -- (y);
\draw (y) -- (u);
    \draw[dotted] (b) -- (y);
    \draw[dotted, thick] (z) to[out=0,in=0] (b);
\end{tikzpicture}
\caption{Structure of the face containing $u$ in \cref{lem:first face}.}
\label{fig:firstface}
\end{figure}
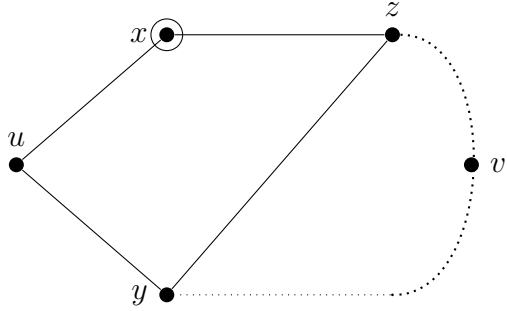

\begin{proposition}
\label{prop:triangle}
{If a graph $G$ is a minimal counterexample to the alternate bunkbed conjecture for forests}, then $G$ cannot contain any triangles. 
\end{proposition}

\begin{proof}
Suppose we have a triangle $xyz$. There are aprori eight possible colourings of the three edges $xy, yz, zx$, but the acyclicity condition rules out the two cases when all three edges have the same colours. We now have three sets of pairs of configurations, the first being the case when $xy$ has a different colour from $yz$ and $xz$, and the other two cases being similar. In all these cases, we can contract an edge (the $xy$ edge in this example) to reduce the problem to a smaller graph. For details, we refer the reader to~\cite{linusson_2011}.
\end{proof}

\begin{lemma}
\label{lem:first face}
{Let $G$ be an outerplanar graph with $T$ being a set of posts, and $u, v$ vertices of $G$. Suppose that $G$  is a minimal counterexample to the alternate bunkbed conjecture \eqref{alt bunkbed forests}. Then the face containing $u$ must be of the special form given in \cref{fig:firstface}. }
\end{lemma}

\begin{proof}
Without loss of generality, we can assume that the face is of the form given in \cref{fig:quadrilateral}.
\begin{figure}[ht]
\centering 
\begin{tikzpicture}
  \node[circle, fill=black, label=above:$u$, inner sep=2pt] (u) at (0,0) {};
      \node[circle, fill=black, label=right:$v$, inner sep=2pt] (v) at (6.05,0) {};
  \node[circle, fill=black, inner sep=0pt] (x) at (2,1.73) {};
  \node[circle, fill=black, label=left:$y$, inner sep=2pt] (y) at (2,-1.73) {};
  \node[circle, fill=black, label=above:$z$, inner sep=2pt] (z) at (5,1.73) {};
  \node[circle, fill=black, label=right:$ $, inner sep=0pt] (b) at (5,-1.73) {};
\draw[dotted] (u) -- (x) -- (z) -- (y);
\draw[dotted] (y) -- (u);
    \draw[dotted] (b) -- (y);
    \draw[dotted, thick] (z) to[out=0,in=0] (b);
\end{tikzpicture}
\caption{Most general structure of the face containing $u$.}
\label{fig:quadrilateral}
\end{figure}
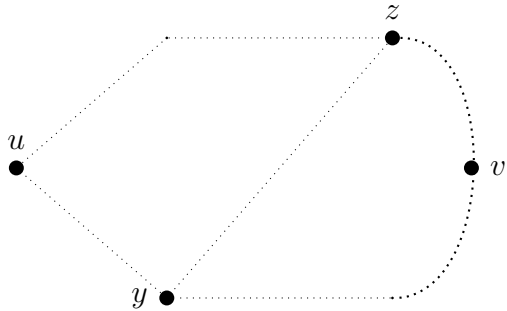

Let $x_1, \dots, x_k$ be the nodes on the path from $u$ to $z$ (call this the \textit{top path}), and $y_1, \dots, y_\ell$ those on the path from $u$ to $y$ (call this the \textit{bottom path}). Statements \cref{thm:linusson}(2f) and (2g) show that one cannot have post vertices in both the top and bottom paths. If neither of them have posts, then the face must be a triangle $uzy$, which is ruled out by \cref{prop:triangle}. Without loss of generality, let us assume that there is a post in the set $\{x_1, ..., x_k, z\}$. \cref{thm:linusson}(2c) and (2d) force us to concede that $k = 1$ and that $x_1$ is a post. This in turn mandates that $\ell = 0$ and we get that the face is of the form in \cref{fig:firstface}.
\end{proof}

\cref{lem:first face} can now be used to prove the following. 

\begin{theorem}
\label{thm:onepostcorrelation}
Let $G$ be an outerplanar graph. Then the {alternate} bunkbed conjecture for forests on $(\tG, T)$ is true whenever $T$ has at most one vertex.
\end{theorem}

\begin{proof}
Let $(G_{0}, T_0)$ be a minimal counterexample and let $u, v$ be the witness vertices for the counterexample. 
Note that $T_0$ has size at most $1$ by \cref{thm:linusson}(2). 
If $T_0$ has size zero, we are done -- there is no path from $u_1$ to $v_2$.

By the preceding arguments, the face containing $u$ must be a quadrilateral containing the post. Now, the same argument holds for the face containing $v$ as well. Since there is a single post, we must necessarily have that this minimal counterexample looks like one of the two graphs in \cref{fig:singlepost}.

\begin{figure}[ht]
\centering 
\begin{tikzpicture}
  \node[circle, fill=black, label=above:$u$, inner sep=2pt] (u) at (0,0) {};
      \node[circle, fill=black, label=right:$v$, inner sep=2pt] (v) at (4.05,0) {};
  \node[circle, fill=black, label=left:$x$, inner sep=2pt] (x) at (2,1.73) {};
  \node[circle, fill=black, label=left:$y$, inner sep=2pt] (y) at (2,-1.73) {};
\draw (2,1.73) circle (6pt);
\draw (u) -- (x) -- (v) -- (y);
\draw (y) -- (u);
\end{tikzpicture}
\hspace{0.3in}
\begin{tikzpicture}
  \node[circle, fill=black, label=above:$u$, inner sep=2pt] (u) at (0,0) {};
      \node[circle, fill=black, label=right:$y$, inner sep=2pt] (y) at (4.05,0) {};
  \node[circle, fill=black, label=left:$x$, inner sep=2pt] (x) at (2,1.73) {};
  \node[circle, fill=black, label=left:$v$, inner sep=2pt] (v) at (2,-1.73) {};
\draw (2,1.73) circle (6pt);
\draw (u) -- (x) -- (y) -- (v);
\draw (v) -- (u);
\end{tikzpicture}
\caption{Structure of the minimal counterexample for \cref{thm:onepostcorrelation}.}
\label{fig:singlepost}
\end{figure}
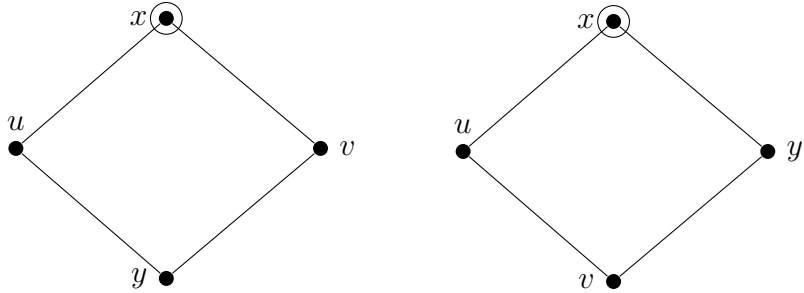

There are 14 possible red-blue colourings under $\altforest{G_0}{T_0}$.
A direct calculation shows that for the left figure, 
\[
\frac{6}{14}  = \altforest{G_0}{T_0}(u \leftrightarrow_{RR} v) > \altforest{G_0}{T_0}(u \leftrightarrow_{RB} v) = \frac{4}{14},
\]
and for the right figure,
\[
\frac{8}{14}  = \altforest{G_0}{T_0}(u \leftrightarrow_{RR} v) > \altforest{G_0}{T_0}(u \leftrightarrow_{RB} v) = \frac{2}{14}.
\]
Therefore, neither of them are counterexamples, giving a contradiction.
\end{proof}

This has the following interesting consequence. A little thought shows that if the bunkbed conjecture for forests is true for outerplanar graphs with a single post, one has the following correlation inequality. 

\begin{corollary}
\label{cor:corr ineq}
Let $G$ be an outerplanar graph, and $u, v, t$ be vertices in $G$. Then 
\[
\forest{G}(u \leftrightarrow v) \geq \forest{G}(u \leftrightarrow t) \, \forest{G}(t \leftrightarrow v).
\]
\end{corollary}

The inequality in \cref{cor:corr ineq} is well known to hold in the case of percolation. Indeed, we have that 
\[
\perc{G}(u \leftrightarrow v) \geq \perc{G}(u \leftrightarrow t, \, t \leftrightarrow v) \geq \perc{G}(u \leftrightarrow t) \, \perc{G}(t \leftrightarrow v),
\]
where the second inequality is the Harris inequality given in a more general form in \cref{thm:harris ineq}. 
The latter inequality also holds for the random cluster measure $\rc{G}$ with $q \geq 1$ due to the FKG inequality~\cite[Theorem~(3.8)]{grimmett-2006}. 
Such arguments are invalid in the regime $q < 1$, and that includes the case of the arboreal gas measure. But it is interesting that a weaker inequality does hold, at least for outerplanar graphs. We will have more to say about this in \cref{sec:corr}.

\subsection{Outerplanar graphs with two posts}
We will now turn to proving the following stronger theorem.

\begin{theorem}
\label{thm:twopostcorrelation}
Let $G$ be an outerplanar graph. Then, the bunkbed conjecture on $(\tG, T)$ with the arboreal gas measure is true whenever the set of posts $T$ has size $2$.
\end{theorem}

\begin{proof}
From \cref{prop:triangle} we know that a minimal counterexample cannot contain any triangles and \cref{thm:linusson}(c) shows that any vertex of degree two must be a post. Thus, there must be a 4-cycle with $u$ with one post as a neighbor; similarly for $v$. With only two posts a minimal counterexample must be one of the two graphs $G$ and $H$ shown in \cref{fig:twopost1}. The lower and upper paths in both graphs have $n$ edges for some $n \geq 1$. 

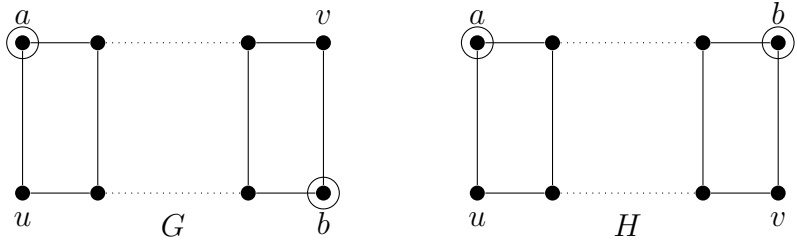
\begin{figure}[ht]
\centering 
\begin{tikzpicture}

 \node[fill=black, circle, inner sep=2pt, label=above:$a$] (a) at (-2,2) {};
 \draw (-2,2) circle (6pt);

  \node[fill=black, circle, inner sep=2pt, label=below:$u$] (u) at (-2,0) {};
   \node[fill=black, circle, inner sep=2pt] (e) at (-1,2) {};
  \node[fill=black, circle, inner sep=2pt] (f) at (-1,0) {};
    \node[fill=black, circle, inner sep=2pt, label=below:$b$] (b) at (2,0) {};
    \draw (2,0) circle (6pt);

  \node[fill=black, circle, inner sep=2pt] (y) at (1,0) {};
  \node[fill=black, circle, inner sep=2pt, label=above:$v$] (v) at (2,2) {};
  \node[fill=black, circle, inner sep=2pt] (x) at (1,2) {};
      \node[label=below:$G$] (G) at (0,0) {};

\draw (a) -- (e) -- (f) -- (u) -- (a);
\draw (v) -- (b) -- (y) -- (x) -- (v);
\draw[dotted] (e) -- (x);
\draw[dotted] (f) -- (y);
        \end{tikzpicture}
        \hspace{0.5in}
        \begin{tikzpicture}

 \node[fill=black, circle, inner sep=2pt, label=above:$a$] (a) at (-2,2) {};
 \draw (-2,2) circle (6pt);

  \node[fill=black, circle, inner sep=2pt, label=below:$u$] (u) at (-2,0) {};
   \node[fill=black, circle, inner sep=2pt] (e) at (-1,2) {};
  \node[fill=black, circle, inner sep=2pt] (f) at (-1,0) {};
    \node[fill=black, circle, inner sep=2pt, label=below:$v$] (v) at (2,0) {};

  \node[fill=black, circle, inner sep=2pt] (y) at (1,0) {};
  \node[fill=black, circle, inner sep=2pt, label=above:$b$] (b) at (2,2) {};
      \draw (2,2) circle (6pt);

  \node[fill=black, circle, inner sep=2pt] (x) at (1,2) {};
      \node[label=below:$H$] (H) at (0,0) {};
\draw (a) -- (e) -- (f) -- (u) -- (a);
\draw (v) -- (y) -- (x) -- (b) -- (v);
\draw[dotted] (e) -- (x);
\draw[dotted] (f) -- (y);
        \end{tikzpicture}
    \caption{Two possible counterexamples in the proof of \cref{thm:twopostcorrelation}.}
    \label{fig:twopost1}
    \end{figure}

The notation $x_1, x_2$ has been used for the two copies of the vertex $x$ in the bunkbed graph. Now we will extend this notation to the red-blue situation to mean a path starting at $x$ with red and blue, respectively. We will use $[x_1y_2\mid z_1]_\ge$ to mean the number of configurations such that there is a admissible path starting in a red edge from $x$ and ending with a blue edge at $y$, but there is no admissible path starting in a red edge from $x$ and ending in $z$ with a red edge. {The subscript $\ge$ indicates that there is no restriction on the number of components in the red-blue coloured graph. 
By considering the relative positions of $u_2$ and $v_2$ in the coloured graph,
we get
\[
[u_1 v_1]_\ge = [u_1v_1u_2\mid v_2]_\ge + [u_1v_1v_2\mid u_2]_\ge + [u_1v_1\mid u_2v_2]_\ge + [u_1v_1\mid u_2\mid v_2]_\ge,
\]
and similarly for $[u_1 v_2]_\ge$. Therefore, we get}
\begin{align*}
[u_1v_1]_\ge-[u_1v_2]_\ge = &[u_1v_1u_2\mid v_2]_\ge-[u_1v_2u_2\mid v_1]_\ge\\
&+ ([u_1v_1\mid u_2v_2]_\ge-[u_1v_2\mid u_2v_1]_\ge) +
([u_1v_1\mid u_2\mid v_2]_\ge-[u_1v_2\mid u_2\mid v_1]_\ge).
\end{align*}
We must show that this quantity is non-negative in both graphs.
By red-blue symmetry, the first two quantities on the right hand side are the same and will cancel each other. 
Therefore we get
\begin{align}
\label{diff}
[u_1v_1]_\ge-[u_1v_2]_\ge = ([u_1v_1\mid u_2v_2]_\ge-[u_1v_2\mid u_2v_1]_\ge)+([u_1v_1\mid u_2\mid v_2]_\ge-[u_1v_2\mid u_2\mid v_1]_\ge).
\end{align}

Let us first look at the first graph $G$. 
We can compute the first term on the right hand side of \eqref{diff}, $[u_1v_1\mid u_2v_2]_\ge-[u_1v_2\mid u_2v_1]_\ge$, exactly. For the first term there is a red-red path from $u$ to $v$ as well as a blue-blue path, the configuration must then be as in \cref{fig:RR and BB}.
We may colour the remaining edges arbitrarily. No cycles will emerge.  There are thus $2^{n-1}$ configurations of either sort and $2^n$ in all. 

\begin{figure}[ht]
\centering 
\begin{tikzpicture}

 \node[fill=black, circle, inner sep=2pt, label=above:$a$] (a) at (-2,2) {};
     \draw (-2,2) circle (6pt);

  \node[fill=black, circle, inner sep=2pt, label=below:$u$] (u) at (-2,0) {};
   \node[fill=black, circle, inner sep=2pt] (e) at (-1,2) {};
  \node[fill=black, circle, inner sep=2pt] (f) at (-1,0) {};
    \node[fill=black, circle, inner sep=2pt, label=below:$b$] (b) at (2,0) {};
        \draw (2,0) circle (6pt);

  \node[fill=black, circle, inner sep=2pt, label=below:$y$] (y) at (1,0) {};
  \node[fill=black, circle, inner sep=2pt, label=above:$v$] (v) at (2,2) {};
  \node[fill=black, circle, inner sep=2pt] (x) at (1,2) {};
\draw [red] (a) -- (e) -- (x) -- (v);
\draw  [blue] (b) -- (y) -- (f) -- (u);
\draw [red] (a) -- (u);
\draw [blue] (v) -- (b);
        \end{tikzpicture}
        \hspace{0.5in}
        \begin{tikzpicture}
 \node[fill=black, circle, inner sep=2pt, label=above:$a$] (a) at (-2,2) {};
     \draw (-2,2) circle (6pt);

  \node[fill=black, circle, inner sep=2pt, label=below:$u$] (u) at (-2,0) {};
   \node[fill=black, circle, inner sep=2pt] (e) at (-1,2) {};
  \node[fill=black, circle, inner sep=2pt] (f) at (-1,0) {};
    \node[fill=black, circle, inner sep=2pt, label=below:$b$] (b) at (2,0) {};
        \draw (2,0) circle (6pt);

  \node[fill=black, circle, inner sep=2pt, label=below:$y$] (y) at (1,0) {};
  \node[fill=black, circle, inner sep=2pt, label=above:$v$] (v) at (2,2) {};
  \node[fill=black, circle, inner sep=2pt] (x) at (1,2) {};
\draw [blue] (a) -- (e) -- (x) -- (v);
\draw  [red] (b) -- (y) -- (f) -- (u);
\draw [blue] (a) -- (u);
\draw [red] (v) -- (b);
        \end{tikzpicture}
        \caption{For the graph $G$, the possibilities when there is both a red-red and a blue-blue path from $u$ to $v$.}
        \label{fig:RR and BB}
        \end{figure}
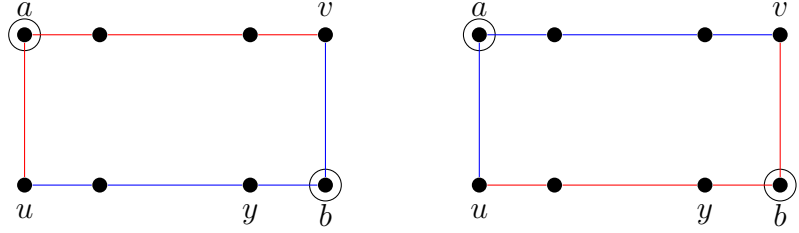
        
For the second term, namely $[u_1v_2\mid u_2v_1]_\ge$, the configurations are of one of the two sorts presented in \cref{fig:RB and BR}. 
If one more edge is red in the left configuration we would have a red-red path from $u$ to $v$, which is not allowed here so all the remaining edges are blue. Similarly all remaining edges in the right configuration are red. There are thus $2$ such configurations. The difference is $2^n - 2$, a non-negative number for all $n\ge 1$. Note that if there is a path between two vertices it must be unique otherwise there would be a cycle.

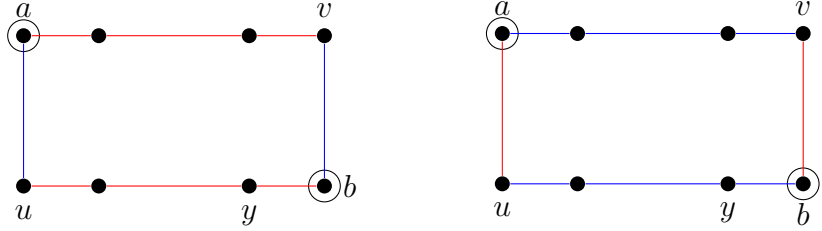
\begin{figure}[ht]
\centering 
\begin{tikzpicture}

 \node[fill=black, circle, inner sep=2pt, label=above:$a$] (a) at (-2,2) {};
     \draw (-2,2) circle (6pt);

  \node[fill=black, circle, inner sep=2pt, label=below:$u$] (u) at (-2,0) {};
   \node[fill=black, circle, inner sep=2pt] (e) at (-1,2) {};
  \node[fill=black, circle, inner sep=2pt] (f) at (-1,0) {};
    \node[fill=black, circle, inner sep=2pt, label=right:$b$] (b) at (2,0) {};
        \draw (2,0) circle (6pt);

  \node[fill=black, circle, inner sep=2pt, label=below:$y$] (y) at (1,0) {};
  \node[fill=black, circle, inner sep=2pt, label=above:$v$] (v) at (2,2) {};
  \node[fill=black, circle, inner sep=2pt] (x) at (1,2) {};
\draw [red] (a) -- (e) -- (x) -- (v);
\draw  [red] (b) -- (y) -- (f) -- (u);
\draw [blue] (a) -- (u);
\draw [blue] (v) -- (b);
        \end{tikzpicture}
        \hspace{0.5in}
        \begin{tikzpicture}
 
 \node[fill=black, circle, inner sep=2pt, label=above:$a$] (a) at (-2,2) {};
     \draw (-2,2) circle (6pt);
     \node[fill=black, circle, inner sep=2pt, label=below:$u$] (u) at (-2,0) {};
   \node[fill=black, circle, inner sep=2pt] (e) at (-1,2) {};
  \node[fill=black, circle, inner sep=2pt] (f) at (-1,0) {};
    \node[fill=black, circle, inner sep=2pt, label=below:$b$] (b) at (2,0) {};
        \draw (2,0) circle (6pt);

  \node[fill=black, circle, inner sep=2pt, label=below:$y$] (y) at (1,0) {};
  \node[fill=black, circle, inner sep=2pt, label=above:$v$] (v) at (2,2) {};
  \node[fill=black, circle, inner sep=2pt] (x) at (1,2) {};
\draw [blue] (a) -- (e) -- (x) -- (v);
\draw  [blue] (b) -- (y) -- (f) -- (u);
\draw [red] (a) -- (u);
\draw [red] (v) -- (b);
        \end{tikzpicture}
        \caption{For the graph $G$, the possibilities when there is both a red-blue and a blue-red path from $u$ to $v$.}
        \label{fig:RB and BR}
        \end{figure}

    Finally, we consider the last two terms on the right hand side of \eqref{diff}.
    We will give an injection from set of configurations contributing to the final term $[u_1v_2\mid u_2\mid v_1]_\ge$ into those contributing to the penultimate term, $[u_1v_1\mid u_2\mid v_2]_\ge$, to prove that their difference is non-negative as well.
    \begin{enumerate}
        \item If the unique red-blue path in $[u_1v_2\mid u_2\mid v_1]_\ge$ starts with the edge $ua$, flip the colours of all edges save $(u, a)$.
        \item If it does not, it must pass through $b$ since it must change color somewhere. Flip the colour of $(b, v)$ and keep all the other edges the same.
    \end{enumerate}
    Note that (1) also covers the case of a path passing through both $a$ and $b$.  We will show that this map is well-defined and an injection into configurations where there is a red - red configuration contributing to $[u_1v_1\mid u_2\mid v_2]_\ge$. 
    
    First, assuming the map is well-defined, we see that the map is an injection: In the first case, the unique path passes through $a$. In the second, it does not. Restricted to each of the two sets, the map is of course an injection: Thus, it is globally an injection as well. 
    
    We now complete the proof by showing the range of the map is $[u_1v_1 \mid u_2 \mid v_2]_\ge$. 
    We will focus on (1). The situation in (2) is very similar.
    First, we show that no cycles will emerge. 
    In (1) if a cycle emerged, it must use the edge $(u, a)$ and the cycle gives a red-red path (not using $(u, a)$) between $u$ and $a$. This means there was a blue-blue path to begin with. But $u_1, u_2$ are not connected, so this is impossible.
    Similarly, $u_1$ and $u_2$ are not connected after the map because, if they were, $u_1$ would have been in a cycle to begin with. By the same argument,
    $v_1$ and $v_2$ are also in different components after the map.
    Finally, in (1), the blue-blue path from $u$ to $v$ after applying the map must come from a red-red path from $u$ to $v$. But, there could have been no such path by assumption. 
    Thus, $u_2$ and $v_2$ are in different components after the map.

    This settles the inequality for $G$, and we note that this is not a bijection: The configurations that will not arise upon applying the map are those with a red-red path from $u$ to $v$ that does not pass through either $a$ or $b$.
        
    For the second graph $H$ in \cref{fig:twopost1}, there are no configurations of the form $[u_1v_2\mid u_2 v_1]_\ge$. This is because there is no room for the two red-blue and blue-red paths to cross. Thus the middle difference is nonnegative. To prove that the last difference $[u_1v_1\mid u_2\mid v_2]_\ge-[u_1v_2\mid u_2\mid v_1]_\ge\ge 0$ we can use the exact same injection as above. The arguments to prove that this is a well-defined injection is in fact the same word for word. 
\end{proof}

\section{Effective resistances and the arboreal bunkbed conjecture for forests with two components}
\label{sec:effec}

As mentioned in \cref{sec:background}, the arboreal gas measure as $\lambda \rightarrow \infty$ converges weakly to the uniform measure on spanning trees. 
The coefficient of the maximal degree term is the number of spanning trees which is the same in $\forest{G}(x\leftrightarrow y)$ independent of the vertices $x,y$. 
When $\lambda$ is very large it is therefore the coefficient of the second largest degree term that determines the inequality.
In this situation the bunkbed inequality is true and we state this as a theorem. 

\begin{theorem}
\label{thm:largel}
Let $G$ be a graph, $\tG$ its bunkbed graph, and $u, v$ vertices in $G$. 
For sufficiently large $\lambda$, we have that 
\[
\forest{\tG} (u_1 \leftrightarrow v_1) \geq 
\forest{\tG}(u_1 \leftrightarrow v_2).
\]
\end{theorem}

\begin{proof}
We need to show that the number of forests in $\tG$ with two components such that $u_1, v_2$ belong to different components is at least the number of forests with $u_1, v_1$ belong to different components. Suppose $G$ has $n$ vertices.
Using the notation from \cref{sec:intro},
\[
\forest{\tG}(u_1 \leftrightarrow v_1) = 
\frac{\lambda^{2n-1} [\cdot ] + \lambda^{2n-2}[u_1 v_1 \mid \cdot] + \cdots}
{\lambda^{2n-1} [\cdot ] + \lambda^{2n-2}[\cdot \mid \cdot] + \cdots},
\]
for large $\lambda$, and similarly for $\forest{\tG}(u_1 \leftrightarrow v_2)$. The denominator is the same in both cases. The desired inequality is equivalent to showing
$[u_1 v_1 \mid \cdot] \geq [u_1 v_2 \mid \cdot]$, which again is equivalent to showing
$[u_1 \mid v_1] \leq [u_1 \mid v_2]$.
Now, from \cref{prop:pseudo}, we need to show that 
\[
R_{\tG}(u_1, v_1) \leq R_{\tG}(u_1, v_2).
\]
Using the definition of the effective resistance in \eqref{resistance distance}, this is equivalent to showing that
\[ 
L^{\dagger}_{\tG}(u_1,v_1) \geq L^{\dagger}_{\tG}(u_1,v_2).
\]
From the structure of the bunkbed graph, the Laplacian can naturally be written as a block matrix in terms of $L_G$ as
\[
L_{\tG} = \left(
\begin{array}{c|c}
L_{G}+I &-I\\
\hline
-I & L_{G}+I
\end{array} \right),
\]
where $I$ is the identity matrix of appropriate size.
\cref{prop:bunkbed pseudoinv} shows that 
\[
L_{\tG}^{\dagger} = \dfrac{1}{2}\left(
\left(
\begin{array}{c|c}
L_{G}^{\dagger}&L_{G}^{\dagger}\\
\hline
L_{G}^{\dagger} & L_{G}^{\dagger}
\end{array} \right)
+
\left(
\begin{array}{c|c}
(L_G+2I)^{-1} & -(L_G+2I)^{-1}\\
\hline
-(L_G+2I)^{-1} & (L_G+2I)^{-1}
\end{array} \right)\right).
\]

The formula for the pseudoinverse of the Laplacian above shows that 
\[ 
L^{\dagger}_{\tG}(u_1,v_1) - L^{\dagger}_{\tG}(u_1,v_2) = (L_G+2I)^{-1}(u, v).
\]
We know that the matrix $L_{\tG}$ is positive semidefinite and thus so is $L_{\tG}+2I$. Additionally, this last matrix is an example of an \emph{$M$ matrix} -- a matrix whose eigenvalues have non-negative real part and whose off diagonal entries are non-positive. It is a well known fact that such matrices have inverses which are entry-wise non-negative, see \cite[Exercise~8.3.P15]{horn-johnson-1985}. Consequently, the last term is non-negative and we are done.
\end{proof}

\begin{remark}
With just a little more work, we can also prove the analogous result for the `bunkbed with posts' problem we considered in \cref{sec:outerplanar}. As before, fix a set of posts $T$ and let $S = V\setminus T$, where $V$ is the set of vertices of $G$. Let us write the Laplacian of $G$ as 
\[
L_G =\begin{blockarray}{c c c}
    & S & T \\
\begin{block}{c[cc]}
  S & L_G^{SS} & L_G^{ST} \\
  T & (L_G^{ST})^{T} & L_G^{TT} \\
\end{block}
\end{blockarray}
\]

We let $\tG_T$ now be the graph where there are two vertices for every vertex in $S$ (as usual when we deal with bunkbed graphs) but we contract the edge between $t_1$ and $t_2$ when $t \in T$. This graph thus has vertices indexed by $S \cup S \cup T$ with the natural edge relations. Now, let  $u, v$ be two vertices in $S$ (if one of more are in $T$, there is nothing to prove). A straightforward but tedious calculation shows that 
\[ 
L^{\dagger}_{\tG_T}(u_1,v_1) - L^{\dagger}_{\tG_T}(u_1,v_2) = (L_G^{SS})^{-1}(u, v).
\]
Once again, it is clear that $L_G^{SS}$ is both invertible and an $M$ matrix, and the final expression is therefore non-negative. 
\end{remark}

\section{Correlation Inequalities}
\label{sec:corr}

Recall that $\rc{G}$ is the random cluster measure on $G$. Suppose we have a bunkbed graph with a single post (say at the vertex $w$). The bunkbed conjecture for this case reduces to showing the correlation inequality 
\begin{equation}\label{weakest}
    \rc{G}(u \leftrightarrow v) \geq \rc{G}(u \leftrightarrow w)\, \rc{G}(w \leftrightarrow v).
\end{equation}
This is open even for the arboreal gas measure. 
We did show in \cref{thm:onepostcorrelation} that this does hold for outerplanar graphs, which is a positive sign. 
In the context of percolation (and indeed for any random cluster measure with $q \geq 1$), the Harris inequality says that we have the stronger statement
\begin{equation}\label{harris}
\perc{G} (u \leftrightarrow w, \, w \leftrightarrow v) \geq \perc{G} (u \leftrightarrow w)\, \perc{G} (w \leftrightarrow v).
\end{equation}

We believe that the arboreal gas version of this inequality also holds. 
\begin{conjecture}
Let $G = (V, E)$ be a graph and $u, v, w \in V$.
Then 
\begin{equation}
\label{harris2}
\forest{G} (u \leftrightarrow w, \, w \leftrightarrow v) \geq \forest{G} (u \leftrightarrow w)\, \forest{G} (w \leftrightarrow v).
\end{equation}
\end{conjecture}

This is too attractive a problem for us to resist saying something about it. When $\lambda$ is very large, this inequality is true for trivial reasons; see the discussion following \eqref{4pt ineq2}. We naturally investigated whether there are even stronger inequalities in the literature, and used them as a testbed for the arboreal gas measure $\forest{G}$ for $\lambda$ large. There is a well-known theorem due to Brooks--Smith--Stone--Tutte~\cite{BSST-1940} in this setting, which we now state. 

Suppose $G = (V, E)$ and $e, f, \in E$.
Orient the edges $e, f$ in an arbitrary but fixed manner. Consider the set of all connected subgraphs of $G$ with exactly $|V|$ edges. Any such graph will contain exactly one cycle. Let $X_{+}$ (resp. $X_-$) be the number of such subgraphs where the cycle contains both $e$ and $f$ and the edges are oriented in the same (resp. opposite) direction in that cycle. 
We will use $[e]$ to denote the number of spanning trees containing $e$ and $[e, f]$ to denote the number of spanning trees containing both $e$ and $f$. The expression $[\cdot ]$ will denote the total number of spanning trees as usual. 

\begin{theorem}[{\cite[Equation (2.34)]{BSST-1940}}]
\label{thm:bsst}
Let $e, f$ be edges in a graph $G$ and $X_\pm$ be defined as above. Then
\[
[e]\, [f] - [\cdot] \, [e, f] = (X_+ - X_{-})^2.
\]
\end{theorem}

This theorem has an equivalent formulation in terms of vertices rather than edges due to Choe~\cite{choe-2008}. 
We note that if we let $e = \{a, b\}$, then $[e] = [a\mid b]$, since there is an explicit bijection from spanning trees containing $e$ to two-component spanning forests in which $a$ and $b$ are different components -- just remove $e$.
Similarly, if we let $f = \{c, d\}$, then 
\[
[e, f] = [a\mid bc\mid d] + [a\mid bd\mid c] + [a\mid cd\mid b] + [c\mid ab\mid d].
\]
We have the following result for more vertices.

\begin{theorem}[{\cite[Theorem 11]{choe-2008}}]
\label{thm:equal}
Let $G = (V, E)$ be a graph and $a, b, c, d \in V$.
\[
[a\mid b][c\mid d] = \Big( [a \mid bc\mid d]+[a\mid bd\mid c] + [b\mid ac\mid d] + [b\mid ad\mid c] \Big) [\cdot ] + \Big( [ad\mid bc] - [ac\mid bd] \Big)^2.
\]
Equivalently, in terms of the arboreal gas measure, 
\begin{multline*} 
\forest{G}(a , b)\,\forest{G}(c , d)=
\Big(\forest{G}(a , b c , d) + 
\forest{G}(a , bd , c) +
\forest{G}(b , ac , d) +\forest{G}(b , ad , c) \Big)\\
 \times \forest{G}(abcd)
+ \Big(\forest{G}(ac , bd)-\forest{G}(ad , bc) \Big)^2,
\end{multline*}
as $\lambda \rightarrow \infty$.
\end{theorem}

We remark here that this theorem is after some algebraic rearrangement, an application of the all minors matrix tree theorem (\cref{thm:all minors}).

What happens for general $\lambda$? The assertion that for all $\lambda$, we have 
\begin{multline}
\forest{G}(a , b)\,\forest{G}(c , d) \geq \\
\Big(\forest{G}(a , b c , d) + \forest{G}(a , bd , c) + 
\forest{G}(b , ac , d) +\forest{G}(b , ad , c) \Big)
 \times \forest{G}(abcd).
\end{multline} 
is a famous open problem due to Kahn~\cite[Conjecture~10.11]{kahn-2000} and popularized by Grimmett and Winkler~\cite{grimmett-winkler-2004}. This is equivalent to saying that edge events in random forests are negatively correlated, 
\begin{align} 
\label{eqn:kgw}
\forest{G}(e)\,\forest{G}(f) \geq
\forest{G}(e, f). 
\end{align}
This problem remains mysterious though, and is far from resolution. We say a little more about this at the end of this section. 

There is another way though, of unspooling \cref{thm:bsst}. This involves keeping track of where $c, d$ lie when $a$ and $b$ are in different components, and similarly for where $a, b$ lie when $c$ and $d$ are in different components. In the limit $\lambda = \infty$, we recover a weak version of \cref{thm:bsst}. 

\begin{theorem}
\label{thm:four-point-correlation}
Let $G = (V, E)$ be a graph and $a, b, c, d \in V$.
Then for large enough $\lambda$, we have the inequality
\begin{multline}
\label{4pt ineq}
 \big( 
   \forest{G}(a , bcd) + \forest{G}(acd, b) + \forest{G}(ac, bd) + \forest{G}(ad, bc) 
 \big) \\
\times \big(
   \forest{G}(c , abd) + \forest{G}(abc, d) + \forest{G}(ac, bd) + \forest{G}(ad, bc)
 \big) \\
\le 
 \big( \forest{G}(a , b c , d) + \forest{G}(a , bd , c) + 
   \forest{G}(b , ac , d) + \forest{G}(b , ad , c)
 \big)\,\forest{G}(abcd) \\
+ \big( \forest{G}(ac , b d) - \forest{G}(ad , bc) \big)^{2}.
\end{multline}

\end{theorem}

This is a positive correlation inequality. Note that the regular positive correlation inequality for large $\lambda$,
\begin{equation}
\label{4pt ineq2}
\forest{G}(a \not\!\leftrightarrow b)\,\forest{G}(c \not\!\leftrightarrow d) \leq
\forest{G}(a \not\!\leftrightarrow b \;\text{and}\; c\not\!\leftrightarrow d ),
\end{equation}
 is true for trivial reasons when $\lambda$ is large (after converting to subgraph counts, the right hand side is of the order $\lambda^{2n-3}$ while the left hand side is of the order $\lambda^{2n-4}$). Our correlation inequality in \cref{thm:four-point-correlation} is more refined. 

There is a special version of this theorem about which we can say more. In the case when $b = d$, \cref{thm:four-point-correlation} becomes

\begin{corollary}
\label{cor:three-point}
Let $G = (V, E)$ be a graph and $a, b, c \in V$.
Then for large enough $\lambda$, we have the inequality
\begin{align*} 
\forest{G}(a , b)\,\forest{G}(b , c) \leq& \,\,\forest{G}(a , b , c)\,\,\forest{G}(abc)
+ \left(\forest{G}(ac , b)\right)^2
\end{align*}
Equivalently, we may rewrite this as 
\begin{align*} 
e_2 \left( \forest{G}(a , bc),\,\forest{G}(ab , c),\, \forest{G}(ac , b) \right)
\leq& \,\,\forest{G}(a , b , c)\,\,\forest{G}(abc),
\end{align*}
where we recall that $e_2(x, y, z) = xy + xz + yz$ is the elementary symmetric polynomial of degree $2$.
Further, we have equality in the limit by \cref{thm:bsst}.
\end{corollary}

Interestingly, \cref{cor:three-point} was shown to be true for the percolation measure recently by Gladkov~\cite{gladkov_2024}. 

To prove \cref{thm:four-point-correlation}, we need the following result.

\begin{lemma}
\label{lem:strong rayleigh}
Let $H$ be a connected subgraph of $G$, with $a,b,c,d\in V(H) \subseteq V(G)$. Then we have
\begin{equation}
\label{strong rayleigh}
\left(\dfrac{[a\mid b]^H}{[\cdot ]^H} - \dfrac{[a\mid b]}{[\cdot ]}\right)
\left(\dfrac{[c\mid d]^H}{[\cdot ]^H} - \dfrac{[c\mid d]}{[\cdot ]}\right) 
\geq 
\left(\dfrac{[ad\mid bc]^H-[ac\mid bd]^H}{[\cdot ]^H} - \dfrac{[ad\mid bc]-[ac\mid bd]}{[\cdot ]}\right)^2,
\end{equation}
where the superscript $H$ in our existing notation refers to counts within $H$.
\end{lemma}

The left hand side of \eqref{strong rayleigh} is positive, because successive applications of \cref{thm:rayleigh} shows each of the factors are positive. Therefore, \cref{lem:strong rayleigh} is a quantitative strengthening of Rayleigh's theorem.

\begin{proof}
We will use the Moore--Penrose psuedoinverse of the Laplacian defined in \cref{sec:alg graph}.
Recall from \cref{prop:pseudo} and \eqref{res dist inner product} that 
\[
\frac{[a \mid b]}{[\cdot ]} = \langle L^{\dagger} (e_a - e_b), e_a - e_b\rangle, \quad 
\frac{[c \mid d]}{[\cdot ]} = \langle L^{\dagger} (e_c - e_d), e_c - e_d\rangle.
\]
Let $M = L^{\dagger}_H-L^{\dagger}$. 
The left hand side of \eqref{strong rayleigh} therefore becomes
\[
\langle M (e_a - e_b), e_a - e_b\rangle \, \langle M (e_c - e_d), e_c - e_d \rangle.
\]
Using \cref{lem:offdiag} twice, we see that 
\[
\langle M (e_a - e_b), e_c - e_d\rangle 
=
\dfrac{[ac\mid bd]^H-[ad\mid bc]^H}{[\cdot ]^H} - \dfrac{[ac\mid bd]-[ad\mid bc]}{[\cdot ]}.
\]
Now, since $H$ is a subgraph of $G$, $L - L_H$ is a PSD matrix and hence $L_H \preceq L$. Further, the kernel of $L$ is contained in the kernel of $L_H$. Then, by an argument similar to that sketched in \cref{sec:alg graph}, 
$L^{\dagger} \preceq L_H^{\dagger}$, and therefore $M$ is PSD by definition of the Loewner order (see \cref{sec:alg graph}). 
The desired inequality follows then from the Cauchy-Schwarz inequality since
\[
\langle M (e_a - e_b), e_a - e_b \rangle \, \langle M (e_c - e_d), e_c - e_d\rangle 
\geq 
\left( \langle M (e_a - e_b), e_c - e_d\rangle\right)^2,
\]
completing the proof.
\end{proof}

\begin{proof}[Proof of \cref{thm:four-point-correlation}]
Recall that the notation $[*]_1$ counts the number of forests respecting the partition with one more than the minimum possible number of components.
We then have that 
\[
\mu[a\mid bcd] = \dfrac{\lambda^{n-1}[a\mid bcd]+\lambda^{n-2}[a\mid bcd]_1+O(\lambda^{n-3})}{\lambda^{n-1}[\cdot ]+\lambda^{n-2}[\cdot ]_1+O(\lambda^{n-3})}.
\]
Substituting such expressions in \eqref{4pt ineq} 
and simplifying, we get a rational function with the numerator a polynomial in $\lambda$ of degree $2n-2$. The highest order term cancels out by \cref{thm:equal}. 
To show that the next highest order term is positive, we need to show that
\begin{equation}
\begin{split}
\label{raw}
&\Big([a \mid bcd]+[acd \mid b]+[ac \mid bd]+[ad \mid bc]\Big)
\Big([abc \mid d]_1+[c \mid abd]_1+[ac \mid bd]_1+[ad \mid bc]_1\Big)\\
&+\Big([abc \mid d]+[c \mid abd]+[ac \mid bd]+[ad \mid bc]\Big)\Big([a \mid bcd]_1+[acd \mid b]_1+[ac \mid bd]_1+[ad \mid bc]_1\Big)\\
\leq &\Big([a \mid bc \mid d]+[a \mid bd \mid c]+[b \mid ad \mid c]+[b \mid ac \mid d]\Big)\,[abcd]_1\\
&+\Big([a \mid bc \mid d]_1+[a \mid bd \mid c]_1+[b \mid ad \mid c]_1+[b \mid ac \mid d]_1\Big)\,[abcd]\\
& +2\Big([ac \mid bd]-[ad \mid b c]\Big)\,\Big([ac \mid bd]_1-[ad \mid bc]_1\Big).
\end{split}
\end{equation}

Something miraculous happens here which makes proving \eqref{raw} easy. All expressions of the form $[\ast]_1$ specify where the four vertices $a, b,c, d$ lie as a subgraph in $G$.
For instance, 
\[
[a\mid bcd]_1 = \sum_{H} [a\mid bcd]^{H} \,[\cdot ]^{H^c}, 
 \]
where the sum is over all subgraphs $H$ induced by the two components containing $a, b, c, d$. 
With this in hand, we may write the left hand side of \eqref{raw} as
\begin{align*} 
&\Big([a \mid bcd]+[acd \mid b]+[ac \mid bd]+[ad \mid bc]\Big)\Big([abc \mid d]_1+[c \mid abd]_1+[ac \mid bd]_1+[ad \mid bc]_1\Big)\\
+&\Big([a \mid bcd]_1+[acd \mid b]_1+[ac \mid bd]_1+[ad \mid bc]_1\Big)\Big([abc \mid d]+[c \mid abd]+[ac \mid bd]+[ad \mid bc]\Big)\\
=&\sum_{H}\Big([a \mid bcd]+[acd \mid b]+[ac \mid bd]+[ad \mid bc]\Big) \\
& \hspace*{1cm} \times \Big([abc \mid d]^H+[c \mid abd]^H+[ac \mid bd]^H+[ad \mid bc]^H\Big)\,[\cdot ]^{H^c}\\
&+\sum_H\Big([a \mid bcd]^H+[acd \mid b]^H+[ac \mid bd]^H+[ad \mid bc]^H\Big) \\
& \hspace*{1cm} \times \Big([abc \mid d]+[c \mid abd]+[ac \mid bd]+[ad \mid bc]\Big)\,[\cdot ]^{H^c}\\
=&\sum_H ([a\mid b][c\mid d]^H+[a\mid b]^H[c\mid d])[\cdot ]^{H^{c}}.
\end{align*}
The right hand side of \eqref{raw} becomes
\begin{align*}
&\sum_H \Big([a \mid b c \mid d]+[a \mid b d \mid c]+[b \mid a d \mid c]+[b \mid a c \mid d]\Big)\,\,[abcd]^H[\cdot ]^{H^{c}}\\
+&\sum_H \Big([a \mid b c \mid d]^H+[a \mid b d \mid c]^H+[b \mid a d \mid c]^H+[b \mid a c \mid d]^H\Big)\,\,[abcd][\cdot ]^{H^{c}}\\
+& 2\sum_H \Big([ac \mid b d]-[ad \mid bc]\Big)\Big([ac \mid b d]^H-[ad \mid bc]^H\Big)[\cdot ]^{H^{c}}.
\end{align*}
We will show that we have a term by term inequality. That is, for every $H$ containing $a, b, c, d$, we have that 
\begin{equation}
\begin{split}
\label{fourptH}
[a\mid b]^H[c\mid d]+ [a\mid b][c\mid d]^H
& \leq \Big([a \mid b c \mid d]+[a \mid b d \mid c]+[b \mid a c \mid d]+[b \mid a c \mid d]\Big)\,\,[abcd]^H\\
+& \Big([a \mid b c \mid d]^H+[a \mid b d \mid c]^H+[b \mid a c \mid d]^H+[b \mid a c \mid d]^H\Big)\,\,[abcd]\\
+&2 \Big([ac \mid b d]-[ad \mid bc]\Big)\Big([ac \mid b d]^H-[ad \mid bc]^H\Big).
\end{split}
\end{equation}
This is equivalent to showing that 
\begin{equation}
\begin{split}
\dfrac{[a\!\mid\! b]^H}{[]^H} \cdot \dfrac{[c\!\mid\! d]}{[]} 
+ \dfrac{[a\!\mid\! b]}{[]} \cdot \dfrac{[c\!\mid\! d]^H}{[]^H}\leq
&\dfrac{[a \!\mid\! b c \!\mid\! d] + [a \!\mid\! b d \!\mid\! c] + [b \!\mid\! a c \!\mid\! d] + [b \!\mid\! a d \!\mid\! c]}{[]} \\
+&\dfrac{[a \!\mid\! b c \!\mid\! d]^H + [a \!\mid\! b d \!\mid\! c]^H + [b \!\mid\! a c \!\mid\! d]^H + [b \!\mid\! a d \!\mid\! c]^H}{[]^H} \\
+&2  \dfrac{[ac \!\mid\! b d] - [ad \!\mid\! b c]}{[]} 
   \cdot \dfrac{[ac \!\mid\! b d]^H - [ad \!\mid\! b c]^H}{[]^H}.
\end{split}
\end{equation}
We have by \cref{thm:equal} that 
\begin{equation}
\dfrac{[a \!\mid\! b c \!\mid\! d]+[a \!\mid\! b d \!\mid\! c]+[b 
\!\mid\! a c \!\mid\! d]+[b \!\mid\! a c\! \mid\! d]}{[]}=\dfrac{[a\!\mid\! b]}{[]}\dfrac{[c\!\mid\! d]}{[]}-\left(\dfrac{[ac\! \mid\! b d]-[ad\! \mid \!bc]}{[]}\right)^2
\end{equation}
and similarly for the subgraph $H$. The desired inequality then becomes 
\begin{equation}
\begin{split}
\dfrac{[a\mid b]^H}{[]^H} \cdot \dfrac{[c\mid d]}{[]} 
+ \dfrac{[a\mid b]}{[]} \cdot \dfrac{[c\mid d]^H}{[]^H} \leq& \dfrac{[a\mid b]}{[]}\dfrac{[c\mid d]}{[]}-\left(\dfrac{[ac \mid b d]-[ad \mid bc]}{[]}\right)^2\\
+& \dfrac{[a\mid b]^H}{[]^H}\dfrac{[c\mid d]^H}{[]^H}-\left(\dfrac{[ac \mid b d]^H-[ad \mid bc]^H}{[]^H}\right)^2\\
+& 2 \dfrac{[ac \mid b d] - [ad \mid b c]}{[]} 
   \cdot \dfrac{[ac \mid b d]^H - [ad \mid b c]^H}{[]^H}.
   \end{split}
   \end{equation}
Rearranging, we get that proving \eqref{fourptH} is equivalent to showing that 
\begin{equation*}
\left(\dfrac{[a\mid b]^H}{[\cdot ]^H} - \dfrac{[a\mid b]}{[\cdot ]}\right)\left(\dfrac{[c\mid d]^H}{[\cdot ]^H} - \dfrac{[c\mid d]}{[\cdot ]}\right) \geq 
\left(\dfrac{[ad\mid bc]^H-[ac\mid bd]^H}{[\cdot ]^H} - \dfrac{[ad\mid bc]-[ac\mid bd]}{[\cdot ]}\right)^2,
\end{equation*}
which is proved in \cref{lem:strong rayleigh}, completing the proof.
\end{proof}

Returning to the conjecture of Kahn--Grimmett--Winkler in \eqref{eqn:kgw}, we propose the following strengthening.

\begin{conjecture}
Let $a, b, c, d$ be vertices in a graph $G$.
Then we have that
\begin{align*} 
\forest{G}(a , b)\,\forest{G}(c , d) \geq & \left(\forest{G}(a, bc , d)+\forest{G}(a , bd, c)+\forest{G}(b,  ac, d)+\forest{G}(b , ad , c)\right)\\
&\times \forest{G}(abcd) + \left(\forest{G}(ac , bd)-\forest{G}(ad , bc)\right)^2.
\end{align*}
\end{conjecture}

We have verified this conjecture in \texttt{SageMath} {for arbitrary $\lambda$} by considering the arboreal gas measure on the complete graph $K_n$ for $n \leq 6$, where each edge has an arbitrary symbolic weight {and $\lambda$ is another symbolic variable}. The numerators of the difference of the right hand side and the left hand side is a multivariate polynomial in these weights with positive coefficients. In the case of $n = 6$, there are $15$ variables and the difference has $56,922$ terms.

\section*{Acknowledgements}
{The first author (AA) acknowledges support from SERB Core grant CRG/ 2021/001592 as well as the DST FIST program - 2021 [TPN - 700661]. The second author (SL) was supported by 2022-03875 from VR, the Swedish science council. The third author (MR)  
 was supported in part by the International Centre for Theoretical Sciences (ICTS) for the program -   Discrete integrable systems: difference equations, cluster algebras and probabilistic models (code: ICTS/disdecap2024/10)
}

\bibliographystyle{plain}
\bibliography{main}

\appendix

\section{Computational results for other values of $q$}
\label{sec:app}

{We prove \cref{thm:bunkbed rc false} here.}
We work with the counterexample to the bunkbed conjecture for hypergraphs due to Hollom~\cite{hollom_2025}. While this example can be simplified somewhat (two of the hyperedges can be replaced by regular edges and still yield a counterexample), we work with the original for simplicity. The hypergraph is shown in \cref{fig:Hollom}.
The two points $u, v$ which negate the bunkbed conjecture are respectively the vertices $1$ and $10$. The vertices $3, 5, 8$ in the graph are posts. 

\begin{figure}[ht]
\centering 
\begin{tikzpicture}
      \node[circle, fill=black, label=right:$10$, inner sep=2pt] (10) at (9.05,0) {};
      \node[circle, fill=black, label=left:$1$, inner sep=2pt] (1) at (1,0) {};
  \node[circle, fill=black, label=above:$2$, inner sep=2pt] (2) at (2,1.73) {};
    \node[circle, fill=black, label=left:$4$, inner sep=2pt] (4) at (5,1) {};
        \node[circle, fill=black, label=below:$7$, inner sep=2pt] (7) at (5,-1.73) {};

    \node[circle, fill=black, label=left:$6$, inner sep=2pt] (6) at (3.5,0) {};
    \node[circle, fill=black, label=right:$8$, inner sep=2pt] (8) at (6.5,0) {};

  \node[circle, fill=black, label=below:$3$, inner sep=2pt] (3) at (2,-1.73) {};
  \node[circle, fill=black, label=above:$5$, inner sep=2pt] (5) at (8,1.73) {};
    \node[circle, fill=black, label=below:$9$, inner sep=2pt] (9) at (8,-1.73) {};

\draw (2,-1.73) circle (6pt);
\draw (6.5,0) circle (6pt);
\draw (8,1.73) circle (6pt);
\draw[green] (1,0) circle (6pt);
\draw[green] (9.05,0) circle (6pt);

\draw  (2) -- (3);
\draw  (2) -- (5) -- (4) -- (2);
\draw  (4) -- (6) -- (8) -- (4);
\draw  (6) -- (3) -- (7) -- (6);
\draw  (8) -- (7) -- (9) -- (8);
\draw (5) -- (10) -- (9) -- (5);
\filldraw[draw=black,fill=lightgray] (8,1.73) -- (9.05,0) -- (8,-1.73) -- (8,1.73);
\filldraw[draw=black,fill=lightgray] (2,1.73) -- (8,1.73)  -- (5,1) -- (2,1.73);
\filldraw[draw=black,fill=lightgray] (2,1.73) -- (1,0)  -- (2,-1.73) -- (2,1.73);

\filldraw[draw=black,fill=lightgray] (2,-1.73) -- (3.5,0)  -- (5,-1.73)  -- (2,-1.73);
\filldraw[draw=black,fill=lightgray] (6.5,0) -- (5,-1.73)  -- (8,-1.73) -- (6.5,0);
\filldraw[draw=black,fill=lightgray]  (5,1) -- ( (3.5,0)  --  (6.5,0) --  (5,1);

\end{tikzpicture}
\caption{Hollom's counterexample.}
\label{fig:Hollom}
\end{figure}
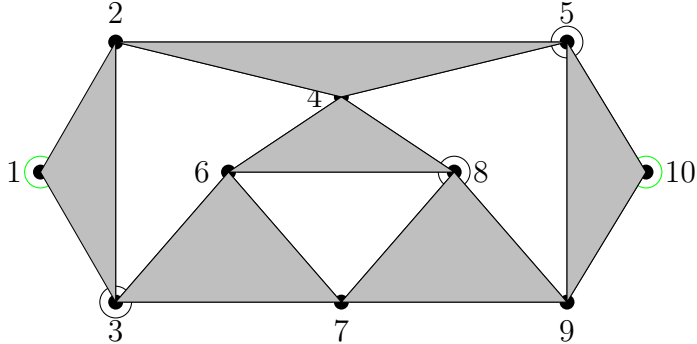

We will next replace each hyperedge by a graph gadget $G_n$ with three distinguished vertices which we label $a, b, c$ and $n$ other vertices as shown in \cref{fig:gadget} (We will see that we can take $n = 31$). We will slot the vertex $a$ onto the vertex in the hyperedge that is a post and the vertices  $b, c$ onto the other two. Since the graph gadget is symmetric in $b, c$, the way they are matched will not matter. We assign weight $p$ to all the horizontal edges in $G_n$ and $1-p$ to each of the other edges, following the lead of Gladkov-Pak-Zimin \cite{gladkov_pak_zimin_2024}.

Let  $p_{abc}$ be the probability (with respect to the random cluster measure on $G_n$) that $a, b, c$ are in the same connected component. The numbers $p_{a\mid bc}, p_{b\mid ac}, p_{c \mid ab}, p_{a\mid b\mid c}$ are defined similarly.

\begin{figure}[ht]
\begin{tikzpicture}
 \node[circle, fill=black, label = above:$a$, inner sep=2pt] (a) at (12,1) {};
    \node[circle, fill=black, label = below:$b$,  inner sep=2pt] (b) at (15,-1) {};
        \node[circle, fill=black,  inner sep=2pt] (c) at (14,-1) {};
    \node[circle, fill=black,  inner sep=2pt] (d) at (13,-1) {};
    \node[circle, fill=black,  inner sep=2pt] (e) at (12,-1) {};

       \node[circle, fill=black,  inner sep=2pt] (f) at (11,-1) {};
    \node[circle, fill=black,  inner sep=2pt] (g) at (10,-1) {};
    \node[circle, fill=black, label = below:$c$, inner sep=2pt] (h) at (9,-1) {};

\draw (12, 1) circle (6pt);

\draw  (a) -- (b) -- (c) -- (d) -- (e) -- (f) -- (g) -- (h) -- (a);
\draw (a) -- (c);
\draw (a) -- (d);
\draw (a) -- (e);
\draw (a) -- (f);
\draw (a) -- (g);

\end{tikzpicture}
\caption{The Gladkov--Pak--Zimin graph gadget~\cite{gladkov_pak_zimin_2024} for $n = 5$.}
\label{fig:gadget}
\end{figure}
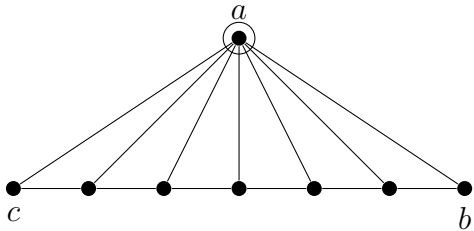

After replacing each hyperedge by a copy of $G_n$, we get a regular graph and a computer assisted calculation\footnote{See the accompanying \texttt{Python} notebook \emph{generalq.ipynb}, where the vertices are numbered $0$ through $19$ as per the usual \texttt{Python} convention.} allows us to compute the polynomials $P_n(p, q)$ where
\[
P_n(p, q):=\mathbb{P}[1 \leftrightarrow 10]-\mathbb{P}[1 \leftrightarrow 20] 
\]
The values of $q > 0$ such that $P_n(1/100, q) < 0$ is presented in \cref{tab:intervals}. The computations were done using rational numbers and are exact.

\begin{table}[h]
\centering
\renewcommand{\arraystretch}{1.2}
\setlength{\tabcolsep}{4pt}

\begin{tabular}{|c|c|c|c|c|c|}
\hline
\textbf{$n$} 
&  3 & 4 & 5 & 6 & 11 \\ \hline
\textbf{Interval} 
& [0.7,\,1.08] & [0.62,\,1.25] & [0.59,\,1.32] & [0.58,\,1.36] 
& [0.56,\,1.42] \\ \hline
\end{tabular}

\vspace{0.3em}

\begin{tabular}{|c|c|c|c|c|c|}
\hline
\textbf{$n$} 
& 21  & 31 & 41  & 51 & 1001 \\ \hline
\textbf{Interval} 
& [0.56,\,1.42]  & [0.56,\,1.43] & [0.56,\,1.43] 
& [0.56,\,1.43] & [0.56,\,1.43]\\ \hline
\end{tabular}

\caption{Truncated intervals (two decimal places) at various indices for $q$ so that $P_n(1/100, q) < 0$.}
\label{tab:intervals}
\end{table}

We will need the following lemma. It says that projected and conditioned random cluster measures can be mutually bounded. 

\begin{lemma}\label{lem:qcomparison}
Let $G$ and $H$ be graphs on the same vertex set and let $m$ be the number of edges in $K = G \cap H$, the graph induced by the edges in both $G$ and $H$. Let $\rc{G \cup H}$ be a random cluster measure on $G \cup H$ and let $\rc{G}$ be the random cluster measure on $G$. Let $r = \max(\{q, q^{-1}\})$. Then for any $A \subset E(G)$, we have that 
\[  
r^{-m} \rc{G}(A) \leq \rc{G \cup H}[\{B \subset E(G \cup H) \mid B \cap E(G) = A\}] 
\leq r^{m} \rc{G}(A).
\]
In particular, the inequality depends on $m$ (the size of $E(K)$), but not otherwise on $H$. 
\end{lemma}

We leave the proof to the interested reader. 

\begin{proof}[Proof of \cref{thm:bunkbed rc false}]
Let us take the hypergraph in \cref{fig:Hollom} and replace each hyperedge by the 
Gladkov--Pak--Zimin graph gadget $G_n$ for say, $n = 31$, where the edge weights are $1/100$ for the horizontal edges and $99/100$ for the other edges. Then, as seen in the \cref{tab:intervals}, we have that $\mathbb{P}[1 \leftrightarrow 10] <\mathbb{P}[1 \leftrightarrow 20]$. This shows that the `bunkbed with posts' conjecture is false in general for these values of $q$. 

This argument can be bootstrapped to an example for the regular bunkbed using the same argument as in \cite[Theorem 6.1]{gladkov_pak_zimin_2024}, at the cost of making the graph larger. We sketch the argument here. 

Referring back to \cref{fig:Hollom}, we now slot in the graph gadgets $G_{31}$ at each hyperedge in the fashion outlined above. 
Call the resulting graph $H$ and the bunkbed of this graph $\widetilde{H}$. For each $t \in T$, attach $k$ `pendant' vertices $u_{t, 1}, \ldots , u_{t, k}$. In the bunkbed graph these will lead to a vertex downstairs and one upstairs, which we call $u_{t,i}^1$ and $u_{t, i}^2$ respectively.  Call the new graph $H_k$ and its bunkbed $\widetilde{H}_k$, and let us use $K_t$ to denote the graph induced by $\{t_1, t_2, u_{t, i}^1, u_{t, i}^2 \mid i \in [k]\}$ for each $t \in T$.

Put the random cluster measure on $\widetilde{H}_k$  and let $\mathcal{T}$ be the random set of posts that are realised. 
Hollom~\cite{hollom_2025} made the following observations. 
\begin{itemize}
    \item Suppose the set of posts in this realization is $S$. If $S \cap V(H) \supsetneq T$ (i.e. the set of posts contains $T$ and at least one more vertex in $H$), then the posts form a cutset in the graph $\widetilde{H}_k$ separating the vertices $1$ and $10$ and the connection probabilities $\mathbb{P}[1 \leftrightarrow 10]$ and $\mathbb{P}[1 \leftrightarrow 20]$ are equal.
    \item Suppose $T \setminus S \neq \emptyset$.
    Let us call a vertex $t$ in $T$ that is not in $S$, but where $t_1$ and $t_2$ are connected within $K_t$, a \textit{quasi-post}. If the set of posts and quasi-posts strictly contains $T$, then once again we have that the connection probabilities are equal. 
\end{itemize}
With this in hand, we see that the connection probabilities can be different only in two cases. 
\begin{itemize}
\item The set of posts in $H$ is exactly $T$ in which case the difference is negative. Call this difference $c$, where $c < 0$ and note that this is independent of $k$. This is because the connectivity between $1$ and $10$, say, only depends on what the posts in $H$ are. The other `pendant' vertices are irrelevant. 
\item At least one of the vertices in $T$ is neither a post nor a quasi-post, in which case the difference could be positive (and clearly at most 1). 
\end{itemize}

An argument by Hollom~\cite[Claim~5.2]{hollom_2025} shows that the probability that $t \in T$ does not become a quasi-post (given that it is not a post) decays exponentially with $k$ in the case of percolation. The same remains true for any $q$ by a simple argument; let it decay as $\alpha(q)^k$ for some fixed $\alpha(q) < 1$, where $\alpha(q)$ only depends on $q$ (and not on $k$). We now apply \cref{lem:qcomparison} to see that inside the graph $H_k$, the probability that the set of posts in $H$ does not contain $T$ can be bounded independent of $k$. Let $r = \max(\{q, q^{-1}\})$. Then, since $\widetilde{H}$ and $\cup_{t \in T}K_t$ meet in $3$ edges,
\[
r^{-3}\,\rc{\widetilde{H}}[\mathcal{T} \cap V(H) \neq T] 
\leq  \rc{\widetilde{H}_k}[\mathcal{T} \cap V(H) \neq T] 
\leq r^{3} \rc{\widetilde{H}}[\mathcal{T} \cap V(H) \neq T]. 
\]
Let 
\[
I_k = \rc{\widetilde{H}_k}[1 \leftrightarrow 10] - \rc{\widetilde{H}_k}[1 \leftrightarrow 20].
\]
We may write
\begin{align*}
I_k &\leq  c\,\rc{\widetilde{H}_k}[\mathcal{T} \cap V(H) = T ]+\rc{\widetilde{H}_k}[\mathcal{T} \cap V(H) \neq T \text{ and not all vertices in }T \\
 & \vspace*{2cm} \text{ are posts or  quasiposts} ]\\
&\leq  c\,\rc{\widetilde{H}_k}[\mathcal{T} \cap V(H) = T ]+\alpha(q)^k \,\rc{\widetilde{H}_k}[\mathcal{T} \cap V(H) \neq T  ]\\
&\leq c\,r^{-3}\,\rc{\widetilde{H}}[\mathcal{T} \cap V(H) = T ]+\alpha(q)^kr^{3}\rc{\widetilde{H}}[\mathcal{T} \cap V(H) \neq T  ].
\end{align*}
Taking $k$ to be appropriately large, we can make $I_k < 0$, since $c < 0$. This yields a counterexample to the regular (weighted) bunkbed problem for $q$ in our range. A counterexample to the regular unweighted bunkbed conjecture can be then created using series-parallel extensions.
\end{proof}

Here is an interesting fact: When we look at the original hypergraph version of the example (without slotting in the graph gadgets), we have that $p_{a\mid bc}, p_{b\mid ac}, p_{c \mid ab}$ are all zero (either the hyperedge is present, which means all the vertices are in the same component) or it is not (which means they are all in distinct components). The polynomial $\mathbb{P}[1 \leftrightarrow 10]-\mathbb{P}[1 \leftrightarrow 20]$ then equals $p_{a\mid b \mid c}^6 p_{abc}^6 \bigg(q^3 - 5q^2 + 10q - 7\bigg) q^5$. The polynomial $q^3 - 5q^2 + 10q - 7$ has a single real root at roughly $1.43$, which is what the right end points of the intervals above seem to converge to. 
\end{document}